\documentclass[a4paper]{amsart} 
\usepackage{amssymb,times, amscd,amsmath,amsthm, xypic}
\usepackage[all]{xy}
\usepackage{geometry}
\usepackage[dvips]{graphicx, color}
\usepackage{mathrsfs}
\author{Shoji Yokura}

\address
{Graduate School of Science and Engineering,
Kagoshima University, 21-35 Korimoto 1-chome, Kagoshima 890-0065, Japan}
\email {yokura@sci.kagoshima-u.ac.jp}
\title 
{A bi-variant algebraic cobordism via correspondences}

\thanks {
\quad \emph{keywords} : (co)bordism, algebraic cobordism, algebraic cobordism of bundles, correspondence }

\numberwithin{equation}{section}
\newtheorem{thm}[equation]{Theorem}
\newtheorem{pro}[equation]{Proposition}
\newtheorem{prob}[equation]{Problem}
\newtheorem{cor}[equation]{Corollary}
\newtheorem{lem}[equation]{Lemma}
\theoremstyle{definition}

\newtheorem{defn}[equation]{Definition}

\newtheorem{rem}[equation]{Remark}

\def\alp{\alpha}
\def\be{\beta}
\def\jeden{1\hskip-3.5pt1}

\def\cal{\mathcal}
\def\ga{\gamma}

\def \bB{\mathbb B}

\def \bM{\mathbb M}

\def\bZ{\mathbb Z}

\def\op{\operatorname}
\def\alp{\alpha}
\def\be{\beta}
\def\jeden{1\hskip-3.5pt1}

\def\Cal{\mathscr}
\def\ga{\gamma}

\def \bB{\mathbb B}

\def \bM{\mathbb M}

\def\bZ{\mathbb Z}

\def\op{\operatorname}

\newcommand{\maru}[1]{\ooalign{
\hfil\resizebox{.8\width}{\height}{#1}\hfil
\crcr
\raise.1ex\hbox{\large$\bigcirc$}}}

\newcommand{\Maru}[1]{\ooalign{
\hfil\resizebox{.6\width}{\height}{#1}\hfil
\crcr
\raise.1ex\hbox{\LARGE$\bigcirc$}}}

\newcommand{\MMaru}[1]{\ooalign{
\hfil\resizebox{.5\width}{\height}{#1}\hfil
\crcr
\raise.1ex\hbox{\Huge$\bigcirc$}}}

\newcommand{\MMMaru}[1]{\ooalign{
\hfil\resizebox{.4\width}{\height}{#1}\hfil
\crcr
\raise.1ex\hbox{\Huge$\bigcirc$}}}

\newcommand{\MMMMaru}[1]{\ooalign{
\hfil\resizebox{.3\width}{\height}{#1}\hfil
\crcr
\raise.1ex\hbox{\Huge$\bigcirc$}}}

\begin{document} 
 
\begin{abstract} 
A bi-variant theory $\mathbb B(X,Y)$ defined for a pair $(X,Y)$ is a theory satisfying properties similar to those of Fulton--MacPherson's bivariant theory $\mathbb B(X \xrightarrow f Y)$  defined for a morphism $f:X \to Y$. In this paper, using correspondences we construct a bi-variant algebraic cobordism $\Omega^{*,\sharp}(X, Y)$ such that $\Omega^{*,\sharp}(X, pt)$ is isomorphic to Lee--Pandharipande's algebraic cobordism of vector bundles $\Omega_{-*,\sharp}(X)$. In particular, $\Omega^{*}(X, pt)=\Omega^{*, 0}(X, pt)$ is isomorphic to Levine--Morel's algebraic cobordism $\Omega_{-*}(X)$. Namely, $\Omega^{*,\sharp}(X, Y)$ is \emph{a bi-variant version} of Lee--Pandharipande's algebraic cobordism of bundles $\Omega_{*,\sharp}(X)$.
\end{abstract} 

\maketitle

\section{Introduction}\label{intro} 

This is a continuation of our previous works \cite{Yo-NYJM} and \cite{Yokura-enriched} 
and also partially motivated by a recent book \cite{GR1} by D. Gaitsgory and N. Rozenblyum. 

In \cite{FM} W. Fulton and R. MacPherson have introduced \emph{bivariant theory} $\mathbb B(X \xrightarrow f Y)$  with an aim to deal with Riemann--Roch type theorems for singular spaces and to unify them. $\mathbb B_*(X):= \mathbb B^{-*}(X \xrightarrow {} pt)$ becomes a covariant functor and $\mathbb B^*(X):= \mathbb B^{*}(X \xrightarrow {\op{id}_X} X)$ 
 a contravariant functor. In this sense  $\mathbb B(X \xrightarrow f Y)$ is called a \emph{bivariant} theory. 
 %%%%%%%%%%%
 In \cite{Yokura-obt} (cf. \cite{Yokura-obt2} and \cite{Yo-EJM}) the author introduced \emph{an oriented bivariant theory} and \emph{a universal oriented bivariant theory} in order to eventually construct a bivariant algebraic cobordism $\Omega^*(X \xrightarrow f Y)$ which is supposed to be a bivariant-theoretic version of Levine--Morel's algebraic cobordism\footnote{It is called ``cobordism", but it is a theory of ``bordism".} $\Omega_*(X)$ in such a way that the covariant part $\Omega^{-*}(X \xrightarrow {} pt)$ becomes isomorphic to Levine--Morel's algebraic cobordism $\Omega_*(X)$, thus $\Omega^*(X \xrightarrow {\op{id}_X} X)$ would become really a new contravariant ``cobordism". 
 
Our general universal bivariant theory $\mathbb M^{\mathcal C}_{\mathcal S}(X \xrightarrow f Y)$ defined on a category $\mathscr C$, equipped with a class $\mathcal C$ of \emph{confined} morphisms, a class $\mathcal S$ of \emph{specialized} morphisms and a class $\mathcal Ind$ of independent squares, is defined to be the free abelian group generated by the set of isomorphism classes of confined morphisms $c:W \to X$ such that \emph{the composite of $c$ and $f$ is a specialized morphism}:
\begin{equation}\label{tri}
\xymatrix{
& W \ar[dl]_{\mathcal C \ni \, \, c} \ar[dr]^{f \circ c \, \, \in \mathcal S}_{\circlearrowleft \quad }  &\\
X \ar[rr]_f && Y.
}
\end{equation}
%%%%%%%%%%
Here we note that confined and specialized morphisms are both \emph{closed under composition and base change} and \emph{the identity morphisms are defined to be both confined and specialized}. In the case of the category of complex algebraic varieties, a proper morphism is confined and a smooth morphism is specialized and in \cite{Yokura-obt} we consider the universal bivariant theory $\mathbb M^{prop}_{sm}(X \xrightarrow f Y)$ where $prop=\mathcal C$ is the class of proper morphisms and $sm =\mathcal S$ the class of smooth morphisms, in order to construct the above-mentioned bivariant algebraic cobordism $\Omega^*(X \xrightarrow f Y)$. If the coproduct $\sqcup$ is well-defined in a category $\Cal C$, then we can consider the Grothendieck group $\mathbb M^{\mathcal C}_{\mathcal S}(X \xrightarrow f Y)^+$, which is the quotient group $\mathbb M^{\mathcal C}_{\mathcal S}(X \xrightarrow f Y)$ modulo the relations $[W \xrightarrow h X] +[V \xrightarrow k X] = [W \sqcup V \xrightarrow {h \sqcup k} X]$.

In \cite{An} (also see his Ph.D Thesis \cite{An-Th}) T. Annala considered the universal bivariant theory  $\mathbb M^{prop}_{qusm}(X \xrightarrow f Y)^+ \otimes \mathbb L$ where $\mathbb L$ is the Lazard ring and $qusm = \mathcal S$ is the class of  quasi-smooth morphisms in the category of derived schemes, instead of considering the above $\mathbb M^{pop}_{sm}(X \xrightarrow f Y)^+$. A quasi-smooth morphism is closed under composition and base change. Note that a local complete intersection (abbr., $\ell.c.i.$) morphism is \emph{not necessarily closed under base change} in the classical category of algebraic schemes, but it \emph{is closed under base change in the category of derived schemes} since it is quasi-smooth. Then, by considering some bivariant ideal $\langle R^{\op{LS}} \rangle (X \xrightarrow f Y)$ of $\mathbb M^{prop}_{qusm}(X \xrightarrow f Y)^+ $, he defined the quotient group
\begin{equation}
\frac {\mathbb {M}_{qusm}^{prop}(X \xrightarrow f Y)^+ \otimes \mathbb L}{\langle R^{\op{LS}} \rangle ((X \xrightarrow f Y) } =:\Omega^*(X \xrightarrow f Y),
\end{equation}
which is his \emph{bivariant derived algebraic cobordism $\Omega^*(X \xrightarrow f Y)$}. It turns out that $\Omega^*(X \to pt)$ is isomorphic to the derived algebraic cobordism $d\Omega_*(X)$ of Lowrey--Sch\"urg \cite{LS}, which is isomorphic to Levine--Morel's algebraic cobordism $\Omega_*(X)$. Thus T. Annala has succeeded in constructing a bivariant-theoretic version of Levine--Morel's algebraic cobordism.

A motivation of considering $\mathbb M^{prop}_{sm}(X \xrightarrow f Y)$, i.e., the form (\ref{tri}) above with $prop=\mathcal C$ and $sm=\mathcal S$,  is that cobordism cycles, which are key ingredients of Levine--Morel's algebraic cobordism $\Omega_*(X)$, are of the form $[W \xrightarrow p X; L_1, \cdots, L_r]$ where $p:W \to X$ is a \emph{proper} morphism, $W$ is \emph{smooth} and $L_i$'s are line bundles over $W$. Namely, putting aside line bundles $L_i$'s, the data of \emph{a proper morphism $W \xrightarrow p X$ with smooth $W$} is the same as the above commutative diagram (\ref{tri}) with $Y=pt$ a point, $f=a_X:X \to pt$ and $a_W=a_X \circ p:W \to pt$, the morphisms to a point, such that $c:W \to X$ is \emph{proper} and $a_W:W \to pt$ is \emph{smooth}. Namely, it is the left commutative diagram below:
%%%%%%%%%%%
\begin{equation}\label{ob-0}
\xymatrix{
& W \ar[dl]_p \ar[dr]^{a_W}_{\circlearrowleft \quad } &\\
X \ar[rr]_{a_X} && pt.
} \qquad \Longleftrightarrow \qquad 
\xymatrix{
& W \ar[dl]_p \ar[dr]^{a_W} &\\
X && pt.
} 
\end{equation}
%%%%%%%%%%
However, it is clear that the left commutative diagram is actually the same as the right diagram (without making it commutative via the map $a_X:X \to pt$). A general diagram $X \xleftarrow p W \xrightarrow s Y$ is called a \emph{correspondence}. Then the free abelian group generated by the isomorphism classes of such correspondences  $X \xleftarrow p W \xrightarrow s Y$ with proper $p$ and smooth $s$ is denoted by $\mathbb M^{prop}_{sm}(X, Y)$, similarly to $\mathbb M^{prop}_{sm}(X \to Y)$. When $Y=pt$ is a point, we have the canonical isomorphism $\mathbb M^{prop}_{sm}(X \to pt) \cong \mathbb M^{prop}_{sm}(X, pt)$. In a general situation we have ``forgetting the morphism $f$", $\frak F: \mathbb M^{prop}_{sm}(X \xrightarrow f Y) \hookrightarrow \mathbb M^{prop}_{sm}(X, Y)$, which is an embedding or a monomorphism, defined by $\frak F([W \xrightarrow p X]):=[X \xleftarrow p W \xrightarrow {f \circ p} Y]$:
\begin{equation}\label{forget-1}
\xymatrix{
& W \ar[dl]_p \ar[dr]^{f \circ p}_{\circlearrowleft \quad }  &\\
X \ar[rr]_f &&  Y
} \overset {\text{forget the morphism $f$ }}{======\Longrightarrow } 
\xymatrix{
& W \ar[dl]_p \ar[dr]^{f \circ p} &\\
X  &&  Y
}
\end{equation}
%%%%%%%%%%%%%%%
Here we note that if $f:X \to Y$ is not surjective, then there does not exist a proper morphism $p:W \to X$ such that the composite $f \circ p:W \to Y$ is smooth, thus $\mathbb M^{prop}_{sm}(X \xrightarrow f Y)=0$, however $\mathbb M^{prop}_{sm}(X, Y)  \not =0$. As explained above,  in \cite{An} Annala overcomes this kind of drawback $\mathbb M^{prop}_{sm}(X \xrightarrow f Y)=0$ by considering  \emph{quasi-smooth morphisms in the category of derived schemes}.
%%%%%%%%%

In \cite{AY} Annala and the author constructed a bivariant theory $\Omega^{*, \sharp}(X \xrightarrow f Y)$, which is generated by the isomorphism classes of $[W \xrightarrow p X; E]$ where $p$ is proper and the composite $f \circ p$ is quasi-smooth and $E$ is a vector bundle over $W$. $\sharp$ refers to the rank of a vector bundle $E$. $\Omega^{*, \sharp}(X \to pt)$ is isomorphic to Lee--Pandharipande's algebraic cobordism of vector bundles $\Omega_{*,\sharp}(X)$ \cite{LP}. Thus  $\Omega^{*, \sharp}(X \xrightarrow f Y)$ is a bivariant-theoretic version of $\Omega_{*,\sharp}(X)$.
%%%%%%%

In \S \ref{bi-variant} we show that there is a \emph{bi-variant}\footnote{In order not to be confused with a \emph{bivariant} theory in the sense of Fulton--MacPherson, we use ``bi-variant" in \cite{Yo-NYJM}.} algebraic cobordim $\Omega^{*, \sharp}(X,Y)$, using isomorphism classes of correspondences $[X \xleftarrow p W \xrightarrow s Y; E]$ with vector bundles $E$ over $W$, and we have the canonical homomorphism``forgetting the morphism $f$":
$$\frak F: \Omega^{*, \sharp}(X \xrightarrow f Y) \to \Omega^{*, \sharp}(X,Y)$$ 
defined by $\frak F([W \xrightarrow p X; E]):=[X \xleftarrow p W \xrightarrow {f \circ p} Y; E]$.
$\Omega^{*, \sharp}(X,pt)$ is isomorphic to Lee--Pandharipande's algebraic cobordism $\Omega_{*,\sharp}(X)$, hence $\Omega^{*, \sharp}(X,Y)$ is \emph{a bi-variant version} of Lee--Pandharipande's algebraic cobordism $\Omega_{*,\sharp}(X)$. In the case when we consider $E=\bf 0$ the zero bundles, $\Omega^{*}(X,Y) = \Omega^{*, 0}(X,Y) $ is \emph{a bi-variant version} of Levine--Morel's algebraic cobordism $\Omega_{*}(X)$.
%%%%%%%%%%

For considering the above bi-variant theory $\Omega (X,Y)$, there are a few more motivations other than the above very simple observation  (\ref{ob-0}):
\begin{enumerate}
\item Speaking of a bivariant theory, there is another bivariant theory, which is called a bivariant $K$-theory or $KK$-theory $K(X,Y)$ due to G. Kasparov \cite{Ka}. This is well-known and has been studied very well in $C^*$-algebra and operator theories. In \cite{EM3} (cf. \cite{CS0, CS}) H. Emerson and R. Meyer described $K(X,Y)$, using correspondences. \emph{Hence it would be reasonable or natural to consider such a $KK$-theory type ``bivariant" theory $\Omega(X,Y)$, instead of $\Omega(X \to Y)$ in the sense of Fulton--MacPherson.}
%%%%%%%%%%
\item Since we have the natural homomorphism $\frak F: \mathbb M^{prop}_{qusm}(X \xrightarrow f Y)^+ \to \mathbb M^{prop}_{qusm}(X, Y)^+$, it would be quite natural to think that there must be some ``bi-variant" version $\Omega^{*, \sharp}(X,Y)$, using $\mathbb M^{prop}_{qusm}(X, Y)^+$, corresponding to $\Omega^{*, \sharp}(X \xrightarrow f Y)$ which is constructed by using $\mathbb M^{prop}_{qusm}(X \xrightarrow f Y)^+$:
$$
\xymatrix
{
\mathbb M^{prop}_{qusm}(X \xrightarrow f Y)^+ \ar[rr]^{\frak F}  \ar[d] && \mathbb M^{prop}_{qusm}(X, Y)^+ \ar[d] \\
\Omega^{*, \sharp}(X \xrightarrow f Y) \ar@{~>}[d] \ar[rr]^{\frak F} && \Omega^{*, \sharp}(X,Y) \ar@{~>}[d]\\
\Omega^{*, \sharp}(X \to pt) \ar[rr]_{\cong}^{\frak F} && \Omega^{*, \sharp}(X,pt). 
}
$$
The bottom isomorphism may be expected from (\ref{ob-0}) above.
%%%%%%%%%%%
\item Levine--Morel's algebraic cobordism $\Omega_*(X)$ satisfies the following:
\begin{itemize}
\item For a proper morphism $f:X \to Y$ we have the pushforward $f_*:\Omega_*(X) \to \Omega_*(Y)$, which is covariantly functorial,
\item For a smooth morphism $g:X \to Y$ we have the pullback $g^*:\Omega_*(Y) \to \Omega_*(X)$, which is contravariantly functorial, 
\item For a fiber square 
$$\CD
X' @> g' >> X \\
@V f' VV @VV f V\\
Y' @>> g > Y, \endCD
$$
where $f$ is proper and $g$ is smooth, thus so are $f'$ and $g'$, respectively, we have 
$$f'_* \circ (g')^* = g^* \circ f_*:\Omega_*(X) \to \Omega_*(Y')$$
which is called \emph{base change isomorphism} or \emph{Beck--Chevalley condition}.
\end{itemize}
Such a functor is called is a ``bi-variant" functor, e.g., see \cite[Part III Categories of Correspondences, \S 1.1, p.271 and \S 0.1.1., p.285]{GR1}. In order to encode such a bi-variant functor, in \cite{GR1} D. Gaitsgory and N. Rozenblyum consider a category $\mathcal C$ (with finite limits) equipped with two classes of morphisms $vert$ and $horiz$ (both \emph{closed under composition and base change}) and consider  correspondences written as follows:
\begin{equation}\label{v-h}
\xymatrix{
c_{0,1}\ar[d]_f \ar[r]^g & c_0\\
c_1
}
\end{equation}
where $f$ is $vert$ and $g$ is $horiz$. For such a correspondence they consider a pushforward $\Phi(f): \Phi(c_{0,1}) \to \Phi(c_1)$ and a pullback $\Phi^!(g): \Phi(c_0) \to \Phi(c_{0,1})$ for some bi-variant functor $\Phi$. Using these, they consider the composition $\Phi(f) \circ \Phi^!(g): \Phi(c_0) \to \Phi(c_1)$, which is considered as a functor defined on the category of correspondences, i.e., a morphism from $c_0$ to $c_1$ is this correspondence. They use such correspondences in order to encode a bi-variant functor with values in a $\infty$-category and furthermore to encode Grothendieck six-functor formalism. As the above diagram (\ref{v-h}) indicates, $vert$ and $horiz$ clearly come from ``vertical" and ``horizontal", respectively, however their real roles or properties are exactly the same as those of \emph{confined} and \emph{specialized}, respectively. Hence, it would be \emph{reasonable or natural to consider some theories similar to a bivariant theory in the sense of Fulton--MacPherson and/or our bi-variant theory in Gaitsgory-Rozenblyum's study of correspondences.} For example, in \cite{Abe} T. Abe introduces a kind of bivariant theory (in the sense of Fulton--MacPherson, but missing the structure of bivariant product) with values in an $\infty$-category, using correspondences.
 \end{enumerate}
%%%%%%%%%

In \S \ref{FM-BT} we make a quick review of Fulton--MacPherson's bivariant theory \cite{FM} and the author's universal bivariant theory \cite{Yokura-obt} (cf. \cite{Yokura-obt2}). In \S \ref{cob-bi} we recall cobordism bicycles of vector bundles $[X \xleftarrow p W \xrightarrow s Y; E]$ and their properties, which are key ingredients for our bi-variant theory $\Omega^{*, \sharp}(X,Y)$. In \S \ref{main} we define our bi-variant theory 
$$\Omega^{*, \sharp}(X,Y):= \frac{\mathcal M^{*, \sharp}(X,Y)^+}{\langle \mathcal R^{\op{LS}} \rangle (X,Y)},$$
after defining a bi-variant ideal $\langle \mathcal R^{\op{LS}}\rangle (X,Y)$, 
such that  $\Omega^{*, \sharp}(X,pt) \cong \Omega_{-*,\sharp}(X)$ which is Lee--Pandharipande's algebraic cobordism of vector bundles \cite{LP}. In the final section \S \ref{dis} we discuss a bit some possible generalized versions of Annala's bivariant algebraic cobordism $\Omega^*(X \to Y)$ and our bi-variant theory $\Omega^*(X,Y)$.

%%%%%%%%%%%%%%%%%%%%%%
\section {Fulton--MacPherson's bivariant theory and a universal bivariant theory}\label{FM-BT}

We make a quick review of Fulton--MacPherson's bivariant theory \cite {FM} (also see \cite{Fulton-book}) and the author's universal bivariant theory \cite{Yokura-obt} (also see \cite{Yokura-obt2}). 

\subsection{Fulton--MacPherson's bivariant theory} Let $\Cal C$ be a category which has a final object $pt$ and on which the fiber product or fiber square is well-defined. Also we consider a class of morphisms, called ``confined morphisms" (e.g., proper morphisms, projective morphisms, in algebraic geometry), which are \emph{closed under composition and base change and contain all the identity morphisms}, and a class of fiber squares, called ``independent squares" (or ``confined squares", e.g., ``Tor-independent" in algebraic geometry, a fiber square with some extra conditions required on morphisms of the square), which satisfy the following:

(i) if the two inside squares in  
$$\CD
X''@> {h'} >> X' @> {g'} >> X \\
@VV {f''}V @VV {f'}V @VV {f}V\\
Y''@>> {h} > Y' @>> {g} > Y \endCD
\quad \quad \qquad \text{or} \qquad \quad \quad 
\CD
X' @>> {h''} > X \\
@V {f'}VV @VV {f}V\\
Y' @>> {h'} > Y \\
@V {g'}VV @VV {g}V \\
Z'  @>> {h} > Z \endCD
$$
are independent, then the outside square is also independent,

(ii) any square of the following forms are independent:
$$
\xymatrix{X \ar[d]_{f} \ar[r]^{\op {id}_X}&  X \ar[d]^f & & X \ar[d]_{\op {id}_X} \ar[r]^f & Y \ar[d]^{\op {id}_Y} \\
Y \ar[r]_{\op {id}_X}  & Y && X \ar[r]_f & Y}
$$
where $f:X \to Y$ is \emph{any} morphism. 

%%%%%%%%%
\begin{rem}Given an independent square, 
its transpose is \emph{not necessarily} independent. For example, let us consider the category of topological spaces and continuous maps. Let \emph{any} map be confined, and we allow a fiber square
$$\CD
X' @> {g'} >> X \\
@V {f'}VV @VV {f}V\\
Y' @>> {g} > Y \endCD
$$ 
to be \emph{independent only if $g$ is proper} (hence $g'$ is also proper). Then its transpose is \emph{not independent unless 
$f$ is proper}. (Note that the pullback of a proper map by any continuous map is proper, because ``proper" is equivalent to ``universally closed", i.e., the pullback by any map is closed.)
\end{rem}
%%%%%%%%

A bivariant theory $\bB$ on a category $\Cal C$ with values in the category of graded abelian groups is an assignment to each morphism
$ X  \xrightarrow{f} Y$
in the category $\Cal C$ a graded abelian group (in most cases we ignore the grading )
$\bB(X  \xrightarrow{f} Y)$
which is equipped with the following three basic operations. The $i$-th component of $\bB(X  \xrightarrow{f} Y)$, $i \in \bZ$, is denoted by $\bB^i(X  \xrightarrow{f} Y)$.
\begin{enumerate}
\item {\bf Product}: For morphisms $f: X \to Y$ and $g: Y
\to Z$, the product operation
$$\bullet: \bB^i( X  \xrightarrow{f}  Y) \otimes \bB^j( Y  \xrightarrow{g}  Z) \to
\bB^{i+j}( X  \xrightarrow{gf}  Z)$$
is  defined.

\item {\bf Pushforward}: For morphisms $f: X \to Y$
and $g: Y \to Z$ with $f$ \emph {confined}, the pushforward operation
$$f_*: \bB^i( X  \xrightarrow{gf} Z) \to \bB^i( Y  \xrightarrow{g}  Z) $$
is  defined.

\item {\bf Pullback} : For an \emph{independent} square \qquad $\CD
X' @> g' >> X \\
@V f' VV @VV f V\\
Y' @>> g > Y, \endCD
$

the pullback operation
$$g^* : \bB^i( X  \xrightarrow{f} Y) \to \bB^i( X'  \xrightarrow{f'} Y') $$
is  defined.
\end{enumerate}
%%%%%%%%%%%%%%%%%
An element $\alp \in \bB(X \xrightarrow f Y)$ is sometimes expressed as follows:
\[
\xymatrix
{
X \ar[rr]_f^{\maru{$\alp$}} && Y
} 
\]
%%%%%%%%%%%%%%%
These three operations are required to satisfy the following seven compatibility 
axioms (\cite [Part I, \S 2.2]{FM}):

\begin{enumerate}
\item[($A_1$)] {\bf Product is associative}: for $X \xrightarrow f Y  \xrightarrow g Z \xrightarrow h  W$ with $\alp \in \bB(X \xrightarrow f Y),  \be \in \bB(Y \xrightarrow g Z), \ga \in \bB(Z \xrightarrow h W)$,
$$(\alp \bullet\be) \bullet \ga = \alp \bullet (\be \bullet \ga).$$
\item[($A_2$)] {\bf Pushforward is functorial} : for $X \xrightarrow f Y  \xrightarrow g Z \xrightarrow h  W$ with $f$ and $g$ confined and $\alp \in \bB(X \xrightarrow {h\circ g\circ f} W)$
$$(g\circ f)_* (\alp) = g_*(f_*(\alp)).$$
\item[($A_3$)] {\bf Pullback is functorial}: given independent squares
$$\CD
X''@> {h'} >> X' @> {g'} >> X \\
@VV {f''}V @VV {f'}V @VV {f}V\\
Y''@>> {h} > Y' @>> {g} > Y \endCD
$$
$$(g \circ h)^* = h^* \circ g^*.$$
\item[($A_{12}$)] {\bf Product and pushforward commute}: for $X \xrightarrow f Y  \xrightarrow g Z \xrightarrow h  W$ with $f$ confined and $\alp \in \bB(X \xrightarrow {g \circ f} Z),  \be \in \bB(Z \xrightarrow h W)$,
$$f_*(\alp \bullet\be)  = f_*(\alp) \bullet \be.$$
\item[($A_{13}$)] {\bf Product and pullback commute}: given independent squares
$$\CD
X' @> {h''} >> X \\
@V {f'}VV @VV {f}V\\
Y' @> {h'} >> Y \\
@V {g'}VV @VV {g}V \\
Z'  @>> {h} > Z \endCD
$$
with $\alp \in \bB(X \xrightarrow {f} Y),  \be \in \bB(Y \xrightarrow g Z)$,
$$h^*(\alp \bullet\be)  = {h'}^*(\alp) \bullet h^*(\be).$$
\item[($A_{23}$)] \label{push-pull}{\bf Pushforward and pullback commute}: given independent squares
$$\CD
X' @> {h''} >> X \\
@V {f'}VV @VV {f}V\\
Y' @> {h'} >> Y \\
@V {g'}VV @VV {g}V \\
Z'  @>> {h} > Z \endCD
$$
with $f$ confined and $\alp \in \bB(X \xrightarrow {g\circ f} Z)$,
$$f'_*(h^*(\alp))  = h^*(f_*(\alp)).$$
\item[($A_{123}$)] {\bf Projection formula}: given an independent square with $g$ confined and $\alp \in \bB(X \xrightarrow {f} Y),  \be \in \bB(Y' \xrightarrow {h \circ g} Z)$
$$\CD
X' @> {g'} >> X \\
@V {f'}VV @VV {f}V\\
Y' @>> {g} > Y @>> h >Z \\
\endCD
$$
and $\alp \in \bB(X \xrightarrow {f} Y),  \be \in \bB(Y' \xrightarrow {h \circ g} Z)$,
$$g'_*(g^*(\alp) \bullet \be)  = \alp \bullet g_*(\be).$$
\end{enumerate}

We also assume that $\bB$ has units:

\underline {Units}: $\bB$ has units, i.e., there is an element $1_X \in \bB^0( X  \xrightarrow{\op {id}_X} X)$ such that $\alp \bullet 1_X = \alp$ for all morphisms $W \to X$ and all $\alp \in \bB(W \to X)$, such that $1_X \bullet \beta = \beta $ for all morphisms $X \to Y$ and all $\beta \in \bB(X \to Y)$, and such that $g^*1_X = 1_{X'}$ for all $g: X' \to X$. 

\underline {Commutativity}\label{commutativity}: $\bB$ is called \emph{commutative} if whenever both
$$\CD
W @> {g'} >> X \\
@V {f'}VV @VV {f}V\\
Y @>> {g} > Z  \\
\endCD  
\quad \quad \text{and} \quad \quad 
\CD
W @> {f'} >> Y \\
@V {g'}VV @VV {g}V\\
X @>> {g} > Z \\
\endCD  
$$
are independent squares with $\alp \in \bB(X \xrightarrow f Z)$ and $\be \in \bB(Y \xrightarrow g Z)$,
$$g^*(\alp) \bullet \be = f^*(\be) \bullet \alp .$$
%%%%%%%%%%%%%%%%%%%%%%%%%
Let $\bB, \bB'$ be two bivariant theories on a category $\Cal C$. A {\it Grothendieck transformation} from $\bB$ to $\bB'$, $\ga : \bB \to \bB'$
is a collection of homomorphisms
$\bB(X \to Y) \to \bB'(X \to Y)$
for a morphism $X \to Y$ in the category $\Cal C$, which preserves the above three basic operations: 
\begin{enumerate}
\item $\ga (\alp \bullet_{\bB} \be) = \ga (\alp) \bullet _{\bB'} \ga (\be)$, 
\item $\ga(f_{*}\alp) = f_*\ga (\alp)$, and 
\item $\ga (g^* \alp) = g^* \ga (\alp)$. 
\end{enumerate}
A bivariant theory unifies both a covariant theory and a contravariant theory in the following sense:

$\bB_*(X):= \bB^{-*}(X \to pt)$ becomes a covariant functor for {\it confined}  morphisms and 

$\bB^*(X) := \bB^*(X  \xrightarrow{id}  X)$ becomes a contravariant functor for {\it any} morphisms. 
%%%%%%%%%%%%%%%%%
\noindent
A Grothendieck transformation $\ga: \bB \to \bB'$ induces natural transformations $\ga_*: \bB_* \to \bB_*'$ and $\ga^*: \bB^* \to {\bB'}^*$.

%%%%% grading %%%%%
\begin{defn}\label{grading}
As to the grading, $\bB_i(X):= \bB^{-i}(X  \to pt)$ and
$\bB^j(X):= \bB^j(X  \xrightarrow{id}  X)$.
\end{defn}
%%%%%%%%%%%%%%%
\begin{defn}\label{canonical}(\cite[Part I, \S 2.6.2 Definition]{FM}) Let $\mathcal S'$ be a class of maps in $\Cal V$, closed under composition\footnote{In the case of confined maps, we require that the pullback of a confined map is confined. For this class $\mathcal S'$ we do not require the stability of pullback. For example, in \cite{FM} the class of local complete intersection (abbr. $\ell.c.i$) morphisms is considered as such a class, and the pullback of an $\ell.c.i.$ morphism is not necessarily a $\ell.c.i.$ morphism.} and containing all identity maps. Suppose that to each $f: X \to Y$ in $\mathcal  S'$ there is assigned an element
$\theta(f) \in \bB(X  \xrightarrow {f} Y)$ satisfying that
\begin{enumerate}
\item [(i)] $\theta (g \circ f) = \theta(f) \bullet \theta(g)$ for all $f:X \to Y$, $g: Y \to Z \in \mathcal S'$ and
\item [(ii)] $\theta(\op {id}_X) = 1_X $ for all $X$ with $1_X \in \bB^*(X):= \bB^*(X  \xrightarrow{\op {id}_X} X)$ the unit element.
\end{enumerate}
Then $\theta(f)$ is called an {\it canonical orientation} of $f$. 
\end{defn} 
%%%%%%%%%%%%%%%%%%%%

%%%%%%%%%%%%
\noindent
{\bf Gysin homomorphisms} \label{gysin}:
Note that such a canonical orientation makes the covariant functor $\bB_*(X)$ a contravariant functor for morphisms in $\mathcal S'$, and also makes the contravariant functor $\bB^*$ a covariant functor for morphisms in $\mathcal  C \cap \mathcal S'$: Indeed, 
\begin{enumerate}
\item As to the covariant functor $\bB_*(X)$: For a morphism $f: X \to Y \in \mathcal S'$ and a canonical orientation $\theta$ on $\mathcal S'$, the following {\it Gysin (pullback) homomorphism}
$$f^!: \bB_*(Y) \to \bB_*(X) \quad \text {defined by} \quad  f^!(\alp) :=\theta(f) \bullet \alp$$
 is {\it contravariantly functorial}. 
\item As to contravariant functor $\bB^*$: For a fiber square (which is an independent square by hypothesis)
$$\CD
X @> f >> Y \\
@V {\op {id}_X} VV @VV {\op {id}_Y}V\\
X @>> f > Y, \endCD
$$
where $f \in \mathcal C \cap  \mathcal S'$, the following {\it Gysin (pushforward) homomorphism}
$$f_!: \bB^*(X) \to \bB^*(Y) \quad \text {defined by} \quad
f_!(\alp) := f_*(\alp \bullet \theta (f))$$
is {\it covariantly functorial}.
\end{enumerate}

 The above Gysin homomorphisms are sometimes denoted by $f_{\theta}^!$ and $f^{\theta}_!$ respectively  in order to make it clear that the canonical orientation $\theta$ is used.

%%%%%%%%%%%%%%%%%
\subsection {A universal bivariant theory }\label{UBT}

\begin{def}\label{nice} (i) Let $\mathcal S$ be another class of morphisms called ``specialized morphisms" (e.g., smooth morphisms in algebraic geometry) in $\Cal C$ , which is \emph{closed under composition and base change and contains all identity morphisms.} 

(ii) Let $\cal S$ be as in (i).
Furthermore, if the orientation $\theta$ on $\cal S$ satisfies that for an independent square with $f \in \cal S$
$$
\CD
X' @> g' >> X\\
@Vf'VV   @VV f V \\
Y' @>> g > Y
\endCD
$$
$\theta (f') = g^* \theta (f)$, i.e., the orientation $\theta$ preserves the pullback operation, then we call $\theta$ a {\it nice canonical orientation} of $\bB$. 
\end{def}
%%%%%%%%%%%%%%%%%%%
%%%%%%%%%%%%%%%%
We also assume that our category $\mathcal V$ satisfies that any fiber square
$$
\CD
P' @> {} >> P\\
@Vf'VV   @VV f V \\
Q' @>> {} > Q
\endCD
$$
with $f$ being confined, i.e., $f \in \mathcal C$, is an independent square. In \cite{Yokura-obt} this condition is called \emph{``$\mathcal C$-independence"}.

\begin{thm}\label{ubt}(\cite[Theorem 3.1]{Yokura-obt})
\label{UBT} \index{universal! bivariant theory} Let  $\Cal C$ be a category with a class $\cal C$ of confined morphisms, a class $\mathcal Ind$  of independent squares and a class $\cal S$ of specialized morphisms.  We define 
$\bM^{\cal C} _{\cal S}(X  \xrightarrow{f}  Y)$
to be the free abelian group generated by the set of isomorphism classes of confined morphisms $h: W \to X$  such that the composite of  $h$ and $f$ is a specialized morphism:
$h \in \cal C$ and $f \circ h: W \to Y \in \cal S.$
\begin{enumerate}
\item  The association $\bM^{\cal C} _{\cal S}$ is a bivariant theory if the three bivariant operations are defined as follows:
\begin{enumerate}
\item  {\bf Product}: For morphisms $f: X \to Y$ and $g: Y
\to Z$, the product 

\hspace{3cm} $\bullet: \bM^{\cal C} _{\cal S} ( X  \xrightarrow{f}  Y) \otimes \bM^{\cal C} _{\cal S} ( Y  \xrightarrow{g}  Z) \to
\bM^{\cal C} _{\cal S} ( X  \xrightarrow{gf}  Z)$

is  defined by
$[V \xrightarrow{h}  X] \bullet [W  \xrightarrow{ k}  Y]:= [V'  \xrightarrow{ h \circ {k}''}  X]$
and extended linearly, where we consider the following fiber squares
$$\CD
V' @> {h'} >> X' @> {f'} >> W \\
@V {{k}''}VV @V {{k}'}VV @V {k}VV\\
V@>> {h} > X @>> {f} > Y @>> {g} > Z .\endCD
$$
\item {\bf Pushforward}: For morphisms $f: X \to Y$
and $g: Y \to Z$ with $f$ 
confined, the pushforward

\hspace{3cm} $f_*: \bM^{\cal C} _{\cal S} ( X  \xrightarrow{gf} Z) \to \bM^{\cal C} _{\cal S} ( Y  \xrightarrow{g}  Z) $

is  defined by
$f_*\left ([V \xrightarrow{h}  X] \right) := [V  \xrightarrow{f \circ h}  Y]$
and extended linearly.

\item {\bf Pullback}: For an independent square
$$\CD
X' @> g' >> X \\
@V f' VV @VV f V\\
Y' @>> g > Y, \endCD
$$
the pullback 

\hspace{3cm} $g^* : \bM^{\cal C} _{\cal S} ( X  \xrightarrow{f} Y) \to \bM^{\cal C} _{\cal S}( X'  \xrightarrow{f'} Y') $

is  defined by
$g^*\left ([V  \xrightarrow{h}  X] \right):=  [V'  \xrightarrow{{h}'}  X']$
and extended linearly, where we consider the following fiber squares:
$$\CD
V' @> g'' >> V \\
@V {{h}'} VV @VV {h}V\\
X' @> g' >> X \\
@V f' VV @VV f V\\
Y' @>> g > Y. \endCD
$$
\end{enumerate}
\item For a specialized morphism $f:X \to Y \in \mathcal S$,

\hspace{3cm} $\theta_{\bM^{\cal C} _{\cal S}}(f) =[X \xrightarrow {\op{id}_X} X] \in \bM^{\cal C} _{\cal S}(X \xrightarrow f Y)$

is a nice canonical orientation of $\bM^{\cal C} _{\cal S}$ for $\mathcal S$.
\item (A universality of $\bM^{\cal C} _{\cal S}$) Let $\bB$ be a bivariant theory on the same category $\Cal C$ with the same class $\cal C$ of confined morphisms, the same class $\mathcal Ind$  of independent squares and the same class $\cal S$ of specialized morphisms, and let $\theta$ be a nice canonical orientation of $\bB$ for $\mathcal S$. Then there exists a unique Grothendieck transformation
$\ga_{\bB} : \bM^{\cal C} _{\cal S} \to \bB$
such that for a specialized morphism $f: X \to Y$, 
$\ga_{\bB} : \bM^{\cal C} _{\cal S}(X  \xrightarrow{f}  Y) \to \bB(X  \xrightarrow{f}  Y)$
satisfies the normalization condition that 
$\ga_{\bB}(\theta_{\bM^{\cal C} _{\cal S}}(f) ) = \theta_{\bB}(f).$
\end{enumerate}
\end{thm}
%%%%%%%%%%%%%%%%%%%%
\begin{pro}\label{comm}(Commutativity\footnote{If $g^*(\alp) \bullet \be = (-1)^{\op{deg}(\alp) \op{deg}(\beta)} f^*(\beta) \bullet \alp$ holds, then it is called \emph{skew-commutative} (see \cite[Part I:Bivariant Theories, \S 2.2]{FM}).})\label{comm} The universal bivariant theory $\bM^{\cal C} _{\cal S}$ is commutative in the sense that $g^*(\alp) \bullet \be = f^*(\beta) \bullet \alp$ for a fiber square
 \[
\xymatrix{X' \ar[rr]^{g'} \ar[d]_{f'} && X \ar[d]_f^{\maru{$\alp$}}   \\
Z'\ar[rr]_g^{\maru{$\be$}} && Z 
} 
\]
\end{pro}
%%%%%%%%%%%%%%%%%%%%
\begin{rem} Here we note that 
\begin{enumerate}
\item for a \emph{confined} morphism $f:X \to X'$ we have the pushforward $f_*:(\bM^{\cal C} _{\cal S})_*(X) \to (\bM^{\cal C} _{\cal S})_*(X')$,
\item for a \emph{specialized} morphism $g:Y \to Y'$ we have the (Gysin) pullback $g^*:(\bM^{\cal C} _{\cal S})_*(Y') \to (\bM^{\cal C} _{\cal S})_*(Y)$, where we use the above canonical orientation $\theta_{\bM^{\cal C} _{\cal S}}(g)$,
\item for a fiber square
$$\CD
X' @> g' >> X \\
@V f' VV @VV f V\\
Y' @>> g > Y, \endCD
$$
where $f$ is confined and $g$ is specialized, thus so are $f'$ and $g'$, respectively. Then we have 
$$f'_* \circ (g')^* = g^* \circ f_*:(\bM^{\cal C} _{\cal S})_*(X) \to (\bM^{\cal C} _{\cal S})_*(Y')$$
which is called \emph{base change isomorphism} or \emph{Beck--Chevalley condition}.
\end{enumerate}
Hence the functor $(\bM^{\cal C} _{\cal S})_*$ is a \emph{bi-variant functor} (e.g., see \cite[Part III Categories of Correspondences, \S 1.1, p.271 and \S 0.1.1., p.285]{GR1})
 from the category $\Cal C$ equipped with two classes of morphisms \emph{confined} and \emph{specialized} to the category of abelian groups. 
\end{rem}
%%%%%%%%%%
\begin{rem}\label{rem-gro} Suppose that the coproduct $\sqcup$ (like the disjoint union in the category of sets, the category of topological spaces, etc.) is well-defined\footnote{Note that there is a category in which the coproduct is not defined, e.g., the category of finite groups.} in the category $\Cal C$ and $f: X_1 \sqcup  X_2 \to Y$ is defined to be confined (resp., specialized) if and only if the restrictions $f|_{X_i}:X_i \to Y$ is confined (resp., specialized). Here $f|_{X_i}:=f \circ \iota_i:X_i \to Y$ where $\iota_i:X_i \to X_1 \sqcup X_2$ is the inclusion. If we consider the following relations on $\mathbb M^{\mathcal C}_{\mathcal S}(X \to Y)$
$$[V_1 \sqcup V_2 \xrightarrow h X]=[V_1 \xrightarrow {h|_{V_1}} X] + [V_2 \xrightarrow {h|_{V_2}} X]$$
where we note that $(h_1 \sqcup h_2)|_{V_i}=(h_1 \sqcup h_2) \circ \iota_i =h_i:V_i \to X$, then we get the Grothendieck group, denoted by $\mathbb M^{\mathcal C}_{\mathcal S}(X \to Y)^+$. Or the set of isomorphism classes of $[V \to X]$ becomes a commutative monoid by the coproduct
$$[V_1 \xrightarrow {h_1} X] + [V_2 \xrightarrow {h_2} X]:=[V_1 \sqcup V_2 \xrightarrow {h_1 \sqcup h_2} X].$$
Then the Grothendieck group $\mathbb M^{\mathcal C}_{\mathcal S}(X \to Y)^+$ is the group completion of this monoid. 
\end{rem}
%%%%%%%%%%%%%%%%%%%%%%%%%
\section{Cobordism bicycles of vector bundles}\label{cob-bi}
In \cite{Yo-NYJM} we considered extending the notion of algebraic cobordism of vector bundles due to Y.-P. Lee and R. Pandharipande \cite{LeeP} (cf. \cite{LTY}) to ``correspondences". For later use, we give a quick recall of them.

\subsection{Cobordism bicycles of vector bundles}
\begin{defn} Let $X \xleftarrow p V \xrightarrow s Y$ be a correspondence such that 
$p$ is \emph{proper} and 
and $s$ is \emph{smooth}, and let $E$ be a complex vector bundle over $V$. Then
$(X \xleftarrow p V \xrightarrow s Y; E)$
is called a \emph{cobordism bicycle of a vector bundle}.
\end{defn}
%%%%%%%%%
%%%%%%%%%%%%
\begin{defn}\label{bicycle} Let $(X \xleftarrow p V \xrightarrow s Y; E)$ and $(X \xleftarrow {p'} V' \xrightarrow {s'} Y; E')$ be two cobordism bicycles of vector bundles of the same rank. If there exists an isomorphism $h: V \cong V'$ such that 
\begin{enumerate}
\item $(X \xleftarrow p V \xrightarrow s Y) \cong (X \xleftarrow {p'} V' \xrightarrow {s'} Y)$ as correspondences, i.e., the following diagrams commute:
$$\xymatrix{
& V\ar [dl]_{p} \ar[dr]^{s} \ar[dd]_h &\\
X   && Y \\
 & V'\ar[ul]^{p'} \ar[ur]_{s'}&}
$$
\item $E \cong h^*E'$,
\end{enumerate}
then they are called isomorphic and denoted by
$$(X \xleftarrow p V \xrightarrow s Y; E) \cong (X \xleftarrow {p'} V' \xrightarrow {s'} Y; E').$$
\end{defn}
%%%%%%%%%%%%
The isomorphism class of 
$(X \xleftarrow p V \xrightarrow s Y; E)$ is denoted by $[X \xleftarrow p V \xrightarrow s Y; E]$, which is still called a cobordism bicycle of a vector bundle.  For a fixed rank $r$ for vector bundles, the set of isomorphism classes of cobordism bicycles of vector bundles for a pair $(X,Y)$  becomes a commutative monoid by the disjoint sum:
$$[X \xleftarrow {p_1} V_1 \xrightarrow {s_1} Y; E_1] + [X \xleftarrow {p_2}  V_2 \xrightarrow {s_2} Y; E_2]
:= [X \xleftarrow {p_1+p_2} V_1 \sqcup V_2 \xrightarrow {s_1+s_2} Y; E_1 \sqcup E_2].$$
This monoid is denoted by $\Cal M^r(X,Y)$ and another grading of $[X \xleftarrow p V \xrightarrow s Y; E]$ is defined by the relative dimension of the smooth morphism $s$, denoted by $\op{dim} s$, thus by double grading, $[X \xleftarrow p V \xrightarrow s Y; E] \in \Cal   M^{n,r}(X,Y)$, where we set $-n = \op{dim} s$ and $r = \op{rank} E$.
The group completion of this monoid, i.e., the Grothendieck group, is denoted by $\Cal   M^{n,r}(X,Y)^+$. 
Using similar notations as in \S \ref{UBT}, $\Cal  M^{n,r}(X,Y)^+$ could or should be denoted by $(\mathbb M^{prop}_{sm})^{n,r}(X,Y)^+$ where $prop=\mathcal C$ is the class of proper morphisms and $sm=\mathcal S$ is the class of smooth morphisms. But, in order to avoid some messy notation, we use 
$\Cal  M^{n,r}(X,Y)^+$, mimicking the notation used in \cite{LM}.
%%%%%%%%%%
\begin{rem}\label{grading} The reason for why we set $-n = \op{dim} s$ in the above grading is requiring that $\Cal M_{-n,r}(X,pt):= \Cal M^{n,r}(X,pt)$ in the case when $Y=pt$ and $\Cal M_{-n,r}(X,pt)$ is supposed to be the same as $\Cal M_{-n,r}(X)^{+}$ considered in Lee--Pandharipande \cite{LP}, where $-n=\op{dim} s$ is nothing but the dimension of the source variety $V$. This is just like setting $\mathbb B_i(X):=\mathbb B^{-i}(X \to pt)$. 
\end{rem}
%%%%%%%%%%
\begin{defn}\label{bi-p}
\label{prod} For three varieties $X,Y,Z$, by using Whitney sum $\oplus$ we define the following product\footnote{In \cite{Yo-NYJM} we consider another product $\bullet_{\otimes}$ using tensor product $\otimes$. In this paper we do not consider $\bullet_{\otimes}$.}
$$\bullet : \Cal   M^{m,r}(X,Y)^+ \otimes \Cal   M^{n,k}(Y,Z)^+ \to \Cal   M^{m+n,r+k}(X,Z)^+$$
$$[X \xleftarrow p  V \xrightarrow s Y; E] \bullet [Y \xleftarrow q  W \xrightarrow t Z; F]
 : = [(X \xleftarrow p V \xrightarrow s Y) \circ (Y \xleftarrow q W \xrightarrow s Z);(q')^*E \oplus (s')^*F ]
 $$
Here we consider the following diagrams: 
$$\xymatrix{
& & (q')^*E \oplus (s')^*F \ar[d] \\
& E \ar[d]& V\times_Y W\ar [dl]_{q'} \ar[dr]^{s'} & F \ar[d]&\\
& V \ar [dl]_{p} \ar [dr]^{s} && W \ar [dl]_{q} \ar[dr]^{t}\\
X & &  Y && Z }
$$
The middle diamond square is the fiber product and
\emph{$(X \xleftarrow p V \xrightarrow s Y) \circ (Y \xleftarrow q W \xrightarrow s Z)$ is the composition of correspondences}, defined by
$(X \xleftarrow p V \xrightarrow s Y) \circ (Y \xleftarrow q W \xrightarrow s Z):= (X \, \, \xleftarrow {p \circ q'}  \, \, V \times_Y W  \, \, \xrightarrow {t \circ s'} \, \, Z)$
in the above diagram.
\end{defn}
%%%%%%%%%%%%%%%%%%%%
For later use, we define the following special cobordism bicycles:
\begin{defn}\label{g-1} 
\begin{enumerate}
\item For a smooth morphism $g:Y \to Y'$, we define
$$\jeden_g:=[Y \xleftarrow {\op{id}_Y} Y \xrightarrow g Y'; {\bf 0}]  \, \,  (\in \Cal M^{- \op{dim}g,0} (Y,Y')^+),$$
\item For a proper morphism $g:Y \to Y'$, we define
$${}_g\jeden:=[Y' \xleftarrow g Y \xrightarrow {\op{id}_{Y}} Y; {\bf 0}] \, \,  (\in \Cal M^{0,0} (Y',Y)^+),$$
\item For any variety $Y$, $\jeden_Y := [Y \xleftarrow {\op{id}_{Y}} Y \xrightarrow {\op{id}_{Y}} Y; {\bf 0}] \, \,  (\in \Cal M^{0,0} (Y,Y)^+).$
\noindent
Here ${\bf 0}$ is the zero bundle over $Y$. 
\item For a vector bundle $E$ over $Y$ we define
$$[[E]]:= [Y \xleftarrow {\op{id}_Y} Y \xrightarrow {\op{id}_Y}  Y; E] .$$
If we need to emphasize the base $Y$ of the bundle $E$, then we denote it by $[[E]]_Y$. Otherwise we omit the suffix $Y$ unless there is some possible confusion on which is the base space of the bundle. Hence for the $0$-bundle over $Y$ we have $[[{\bf 0}]]= \jeden_Y.$
\item For two vector bundles $E$ and $F$ over the same base $Y$, we have
$$[[E]] \bullet [[F]] = [[E \oplus F]].$$
\end{enumerate}
\end{defn}
%%%%%%%%%%
From now on we identify $[X \xleftarrow p V \xrightarrow s Y] = [X \xleftarrow p V \xrightarrow s Y; {\bf 0}]$. Hence we have
$$\jeden_g =[Y \xleftarrow {\op{id}_Y} Y \xrightarrow g Y'] =[Y \xleftarrow {\op{id}_Y} Y \xrightarrow g Y'; {\bf 0}], \quad {}_g\jeden =[Y' \xleftarrow g Y \xrightarrow {\op{id}_{Y}} Y]=[Y' \xleftarrow g Y \xrightarrow {\op{id}_{Y}} Y; {\bf 0}],$$
$$\jeden_Y =[Y \xleftarrow {\op{id}_{Y}} Y \xrightarrow {\op{id}_{Y}} Y]=[Y \xleftarrow {\op{id}_{Y}} Y \xrightarrow {\op{id}_{Y}} Y; {\bf 0}].$$
%%%%%%%%%%%%%%%%%%%%
{\begin{rem}
We note that by Definitions \ref{g-1} and \ref{bi-p} we have
\begin{equation}
[X \xleftarrow p  V \xrightarrow s Y] = {}_p\jeden \bullet \jeden_s.
\end{equation}
In general, we can see
\begin{equation}
[X \xleftarrow p  V \xrightarrow s Y;E] = {}_p\jeden \bullet [[E]] \bullet \jeden_s.
\end{equation}
So, we have that $[X \xleftarrow p  V \xrightarrow s Y; {\bf 0}]= {}_p\jeden \bullet [[{\bf 0}]] \bullet \jeden_s = {}_p\jeden \bullet \jeden_Y \bullet \jeden_s = {}_p\jeden \bullet \jeden_s.$
\end{rem}

%%%%%%%%
%%%%%%%%%%
Now we define pushforward and pullback of cobordism bicycles of vector bundles:
\begin{defn}\label{push-pull}
\begin{enumerate} 
\item (Pushforward) \label{pp-push}
\begin{enumerate}
\item For a \emph{proper} morphism $f:X \to X'$, the (proper) pushforward \emph{acting on the first factor $X \xleftarrow{p} V$}, 
denoted by $f_*:\Cal   M^{m,r}(X,Y)^+ \to \Cal M^{m,r}(X',Y)^+$,  is defined by
$$f_*([X \xleftarrow{p} V  \xrightarrow{s} Y; E]):= [X' \xleftarrow{f \circ p} V  \xrightarrow{s} Y; E].$$
Note that $f_*(-) ={}_f \jeden \bullet (-)$.
\item For a \emph{smooth} morphism $g:Y \to Y'$, the (smooth) pushforward \emph{acting on the second  factor $V  \xrightarrow{s} Y$}, 
denoted by ${}_*g:\Cal   M^{m,r}(X,Y)^+ \to \Cal   M^{m-\op{dim}g,r}(X,Y')^+$ \footnote{In \cite{Yo-NYJM} we have $ \Cal   M^{m+\op{dim}g,r}(X,Y')^+$, instead of $\Cal   M^{m-\op{dim}g,r}(X,Y')^+$, because in \cite{Yo-NYJM} we set $n= \op{dim}s$, not $-n= \op{dim}s$, for the first grading $n$ of $\Cal   M^{n,r}(X,Y)^+$ (cf. Remark \ref{grading} above). This sign change is applied to the other situations below, as long as the relative dimensions of smooth morphisms are involved.}, is defined by
$$([X \xleftarrow{p} V  \xrightarrow{s} Y; E])\, {}_*g:= [X \xleftarrow{p} V  \xrightarrow{g \circ s} Y'; E].$$
Here we emphasize that ${}_*g$ is written on the right side of $([X \xleftarrow{p} V  \xrightarrow{s} Y; E])$ \emph{not on the left side}. 
Note that $(-) {}_*g= (-)\bullet \jeden_g.$

(Note also that $-m = \op{dim} s$ and $\op{dim} (g \circ s) = \op{dim}s + \op{dim}g = -m + \op{dim}g =- (m - \op{dim}g)$.) 
\end{enumerate}
%%%%%%%%%%%%%%%
\item (Pullback) \label{pp-pull}
\begin{enumerate}
\item \label{pp-pull-a} For a \emph{smooth} morphism $f:X' \to X$, the (smooth) pullback \emph{acting on the first factor $X \xleftarrow{p} V$}, 
denoted by $f^*:\Cal   M^{m,r}(X,Y)^+ \to \Cal   M^{m-\op{dim}f,r}(X',Y)^+$ is defined by
$$f^*([X \xleftarrow{p} V  \xrightarrow{s} Y; E]) := [X' \xleftarrow{p'} X' \times_X V \xrightarrow{s \circ f'} Y; (f')^*E].$$
Here we consider the following commutative diagram:
$$\xymatrix{
& (f')^*E \ar[d] 
& E \ar[d] \\
& X' \times_X V \ar[dl]_{p'} \ar[r]^{\qquad f'} & V \ar [dl]^{p} \ar [dr]^{s}\\
X' \ar[r]_f & X &  & Y }
$$
Note that $f^*(-) =\jeden_f \bullet (-)$.

(Note also that the left diamond is a fiber square, thus $f':X'\times_X V \to V$ is smooth and $p':X'\times_X V \to X'$ is proper. Note that 
$\op{dim} f' = \op{dim}f$ and $\op{dim}(s \circ f') = \op{dim} s + \op{dim}f' = -m + \op{dim} f = -(m -\op{dim} f) $.)
\item \label{pp-pull-b} For a \emph{proper} morphism $g:Y' \to Y$, the (proper) pullback \emph{acting on the second  factor $V  \xrightarrow{s} Y$}, 
denoted by ${}^*g:\Cal   M^{m,r}(X,Y)^+ \to \Cal   M^{m,r}(X,Y')^+$, is defined by
$$([X \xleftarrow{p} V  \xrightarrow{s} Y; E])\, {}^*g := [X \xleftarrow{p \circ g'} V \times _Y Y' \xrightarrow{s'} Y'; (g')^*E].$$
Here we consider the following commutative diagram:
$$\xymatrix{
& E \ar[d] & (g')^*E 
\ar[d] \\
& V \ar [dl]^{p} \ar [dr]_{s} & V \times _Y Y' \ar[l]_{g' \quad } \ar[dr]^{s'}\\
X &  & Y & Y' \ar[l]^g}
$$
Note that $(-) {}^*g= (-)\bullet {}_g\jeden$.

(Note that the right diamond is a fiber square, thus $s':V\times_Y Y'\to Y'$ is smooth and $g':V\times_Y Y'\to V$ is proper, and $\op{dim} s = \op{dim} s'$.)
\end{enumerate}
\end{enumerate}
\end{defn}
%%%%%%%%%%%%
%%%%%%%%%%%%
%%%%%%%%%%%%%%%%%%%%
%%%%%%%%%%%%%%%%%%%%
\begin{rem}\label{rem-push-pull} 
Since we refer to the formulas noted in the above Definition \ref{push-pull} later, we list them here again:
\begin{equation}\label{f_*}
f_*(-) ={}_f \jeden \bullet (-),
\end{equation}
\begin{equation}\label{_*}
(-) {}_*g= (-)\bullet \jeden_g,
\end{equation}
\begin{equation}\label{f^*}
f^*(-) =\jeden_f \bullet (-),
\end{equation}
\begin{equation}\label{*g}
(-) {}^*g= (-)\bullet {}_g\jeden.
\end{equation}
\end{rem}

%%%%%%%%
%%%%%%%%%%

\begin{pro}\label{proposition1}
The above three operations of product, pushforward and pullback satisfy the following properties.
\begin{enumerate}
\item[($A_1$)] {\bf Product is associative}: For three varieties $X,Y,Z, W$ we have
$$(\alp \bullet \be) \bullet \ga = \alp \bullet (\be \bullet \ga) 
.$$
where $\alp \in \Cal   M^{m,r}(X,Y)^+, \be \in \Cal   M^{n,k}(Y,Z)^+$ and $\be \in \Cal   M^{\ell,e}(Z,W)^+$, 
%%%%%%%%%%%%
\item[($A_2$)] {\bf Pushforward is functorial} : 
\begin{enumerate}
\item For two proper morphisms $f_1:X \to X', f_2:X' \to X''$, we have
$$(f_2 \circ f_1)_* = (f_2)_* \circ (f_1)_*$$
where  $(f_1)_*:\Cal   M^{m,r}(X,Y)^+ \to \Cal   M^{m,r}(X',Y)^+$ and \\
\hspace{0.9cm} $(f_2)_*:\Cal   M^{m,r}(X',Y)^+ \to \Cal   M^{m,r}(X'',Y)^+$.
%%%%%%%%%%%%%%
\item For two smooth morphisms $g_1:Y \to Y', g_2:Y' \to Y''$ we have
$${}_*(g_2 \circ g_1) = {}_*(g_1) \circ {}_*(g_2)  {}
\, \quad \text{i.e., $\alp\,{}_*(g_2 \circ g_1) = (\alp \, {}_*(g_1)) \, {}_*(g_2)$} $$
where ${}_*(g_1):\Cal   M^{m,r}(X,Y) \to \Cal   M^{m-\op{dim}g_1,r}(X,Y')$ and \\
\hspace{0.9cm}  ${}_*(g_2):\Cal   M^{m-\op{dim}g_1,r}(X,Y') \to \Cal   M^{m-\op{dim}g_1- \op{dim}g_2,r}(X,Y'')$.
\end{enumerate}
%%%%%%%%%%%%%%5
\item[($A_2$)'] {\bf Proper pushforward and smooth pushforward commute}: For a proper morphism $f:X \to X'$ and a smooth morphism $g:Y \to Y'$ we have
$${}_*g \circ f_*= f_* \circ \, {}_*g, \,\,  \text{i.e.,} \, \, (f_*\alp)\, {}_*g = f_*(\alp\, {}_*g) {} 
.$$
%%%%%%%%%%%
\item[($A_3$)] {\bf Pullback is functorial}: 
\begin{enumerate}
\item For two smooth morphisms $f_1:X \to X', f_2:X' \to X''$ we have
$$(f_2 \circ f_1)^* = (f_1)^* \circ (f_2)^*$$
where  $(f_2)^*:\Cal   M^{m,r}(X'',Y)^+ \to \Cal   M^{m - \op{dim}f_2,r}(X',Y)^+$ and \\
\hspace{1cm} $(f_1)^*:\Cal   M^{m -\op{dim}f_2,r}(X',Y)^+ \to \Cal   M^{m-  \op{dim}f_2-\op{dim}f_1,r}(X,Y)^+$. 
\item For two proper morphisms $g_1:Y \to Y', g_2:Y' \to Y''$ we have
$${}^*(g_2 \circ  g_1)  = {}^*(g_2) \circ {}^*(g_1) \, \quad \text{i.e., $\alp\,{}^*(g_2 \circ g_1) = (\alp \, {}^*(g_2)) \, {}^*(g_1)$} $$
where ${}^*(g_1):\Cal   M^{m,r}(X,Y')^+ \to \Cal   M^{m,r}(X,Y)^+$ and \\
\hspace{1cm}  ${}^*(g_2):\Cal   M^{m,r}(X,Y'')^+ \to \Cal   M^{m,r}(X,Y')^+$.
\end{enumerate}
%%%%%%%%%%%%%
\item[($A_3$)'] {\bf Proper pullback and smooth pullback commute}: For a smooth morphism $g:X' \to X$ and a proper morphism $f:Y' \to Y$ we have
$$g^* \circ {}^*f={}^*f \circ g^*,\,\,  \text{i.e.,} \, \, g^*(\alp \, {}^*f) = (g^*\alp)\, {}^*f .$$

%%%%%%%%%%%%
\item[($A_{12}$)] {\bf Product and pushforward commute}: 
Let $\alp \in \Cal M^{m,r}(X,Y)^+$ and $\be \in \Cal   M^{n,k}(Y,Z)^+$.
 \begin{enumerate} 
 %%%%%%%%%
\item For a proper morphism $f:X \to X'$, 
$$f_*(\alp \bullet  \be)  = (f_*\alp) \bullet \be   \quad (\in \Cal   M^{m+n,r+k}(X',Z)^+).$$
%%%%%%%%%%%%% 
\item For a smooth morphism $g:Z \to Z'$, 
$$(\alp \bullet  \be)\, {}_*g  = \alp \bullet (\be\, {}_*g)   \quad (\in \Cal   M^{m+n-\op{dim}g,r+k}(X,Z')^+).$$
\end{enumerate}
%%%%%%%%%%%%
\item[($A_{13}$)] {\bf Product and pullback commute}: Let $\alp \in \Cal   M^{m,r}(X,Y)^+$ and 

\noindent
$\be \in    \Cal   M^{n,k}(Y,Z)^+$.
 
\begin{enumerate} 
\item For a smooth morphism $f:X' \to X$,
$$f^*(\alp \bullet \be)  = (f^*\alp) \bullet \be \quad ( \in \Cal   M^{m+n- \op{dim}f,r+k}(X',Z)^+ ).$$

\item For a proper morphism $g:Z' \to Z$, 
$$(\alp \bullet \be)\, {}^*g = \alp \bullet (\be \, {}^*g) \quad ( \in \Cal   M^{m+n,r+k}(X,Z')^+ ).$$ 
\end{enumerate}
%%%%%%%%%%%%%
\item[($A_{23}$)] {\bf Pushforward and pullback commute}: For $\alp \in \Cal   M^{m,r}(X,Y)^+$
 
\begin{enumerate} 
\item (proper pushforward and proper pullback commute) For proper morphisms $f:X \to X'$ and $g:Y' \to Y$ and for $\alp \in \Cal   M^{m,r}(X,Y)^+$
$$(f_*\alp)\, {}^*g = f_*(\alp\, {}^*g) \quad ( \in \Cal   M^{m,r}(X',Y')^+ ).$$
%%%%%%%%%%%%%%%%
\item (smooth pushforward and smooth pullback commute) For smooth morphisms 
$f:X' \to X$ and $g:Y \to Y'$ and for $\alp \in \Cal   M^{m,r}(X,Y)^+$
$$f^*(\alp\, {}_*g) = (f^*\alp)\, {}_*g \quad ( \in \Cal   M^{m- \op{dim} f - \op{dim}g,r}(X',Y')^+ ).$$
\item (proper pushforward and smooth pullback ``commute" in the following sense) For the following fiber square
 $$\CD
\widetilde X @> {\widetilde f}>> X''\\
@V {\widetilde g} VV @VV {g} V\\
X' @>> {f} > X \endCD
$$ 
with $f$ proper and $g$ smooth, we have
$$g^*f_* = \widetilde f_* \widetilde g^*.$$

\item (smooth pushforward and proper pullback ``commute" in the following sense) For the following fiber square
 $$\CD
\widetilde Y @> {\widetilde f}>> Y''\\
@V {\widetilde g} VV @VV {g} V\\
Y' @>> {f} > Y \endCD
$$ 
with $f$ proper and $g$ smooth, for $\alp \in \Cal   M^{m,r}(X,Y'')^+$ we have
$$(\alp\, {}_*g)\, {}^*f = (\alp\, {}^*\widetilde f)\, {}_*\widetilde g.$$

\end{enumerate}
%%%%%%%%%%%%%%%%
\item[($A_{123}$)] {\bf ``Projection formula"}: 
\begin{enumerate}
\item For a smooth morphism $g:Y \to Y'$ and $\alp \in \Cal   M^{m,r}(X,Y)^+$ and $\be \in \Cal   M^{n,k}(Y', Z)^+$,
$$(\alp\, {}_*g) \bullet \be = \alp \bullet g^*\be \quad ( \in \Cal   M^{m+n- \op{dim}g,r+k}(X,Z)^+ ).$$
\item For a proper morphism $g:Y' \to Y$, $\alp \in \Cal   M^{m,r}(X,Y)^+$ and $\be \in \Cal   M^{n,k}(Y', Z)^+$,
$$(\alp\, {}^*g) \bullet \be  = \alp \bullet g_*\be \quad ( \in \Cal   M^{m+n,r+k}(X,Z)^+ ).$$
\end{enumerate}
\item[(Units)] $\jeden_X = [X \xleftarrow {\op{id}_X} X \xrightarrow {\op{id}_X} X;{\bf 0}]\in \Cal   M^{0, 0}(X,X)^+$ 
satisfies that $\jeden_X \bullet \alp = \alp$ for any element $\alp \in \Cal   M^{m, r}(X, Y)^+$ and $\be \bullet \jeden_X = \be$ for any element $\be \in \Cal   M^{m, r}(Y, X)^+$. 
\end{enumerate}
\end{pro}
%%%%%%%%%%%
\begin{rem}\label{special cobordism}
We note that by Definition \ref{push-pull} (\ref{pp-push}) we have
\begin{equation}\label{p_1}
\text{${}_f \jeden \bullet {}_p \jeden = 
{}_{f \circ p} \jeden$ \quad for proper maps $f$ and $p$,}
\end{equation} 
\begin{equation}\label{1_p}
\text{$\jeden_s \bullet \jeden_g = 
\jeden_{g \circ s} $ \quad for smooth maps $s$ and $g$,}
\end{equation}
\begin{equation}\label{1}
\text{${}_{\op{id}_X} \jeden = \jeden_{\op{id}_X} = \jeden_X$ is the unit. \qquad \qquad \quad }
\end{equation}
This reflects also the \emph{symmetry} between the classes of proper and smooth morphisms, coming from the \emph{transposition} of correspondences,  given by\emph{switching} the two sides of a correspondence. Namely, transposition of ${}_f \jeden \bullet {}_p \jeden =
{}_{f \circ p} \jeden$ is $\jeden_p \bullet \jeden_f = \jeden_{f \circ p}$, thus transposition of (\ref{p_1}) is (\ref{1_p}) and vice versa, putting aside properness and smoothness of maps.
We also see by (\ref{1_p}) that $\jeden_s$ for a smooth map $s$ is exactly a canonical orientation in Definition \ref{canonical}.
And by Definition \ref{push-pull} (\ref{pp-pull}), it follows from both (a) and (b) that for the following fiber square
$$
\xymatrix
{
\widetilde X \ar[r]^{\widetilde g} \ar[d]_{\widetilde s} &  X'' \ar[d]^s\\
X' \ar[r]_g & X
}
$$
with $s$ (hence $\widetilde s$ as well) being smooth and $g$ (hence $\widetilde g$ as well) being proper, we have the following
\begin{equation}\label{s-g}
\jeden_s \bullet {}_g \jeden = {}_{\widetilde g} \jeden \bullet \jeden_{\widetilde s}.
\end{equation} 
By Definition \ref{push-pull} (\ref{pp-pull}) we have the following:
\begin{equation}
\jeden_f \bullet [[E]] =[[f^*E]] \bullet \jeden_f,
\end{equation}
\begin{equation}
[[E]] \bullet {}_g \jeden = {}_g \jeden \bullet [[g^*E]],
\end{equation}
which are kind of ``projection formula".
Hence, using these properties and the associativity of the product $\bullet$, for example, we can do the following computation:
\begin{align*}
[X \xleftarrow p V \xrightarrow s Y;E] \bullet [Y \xleftarrow g W \xrightarrow t Z;F] & = ({}_p \jeden \bullet [[E]] \bullet \jeden_s) \bullet ({}_g \jeden \bullet [[F]] \bullet \jeden_t)\\
 & = {}_p \jeden \bullet [[E]] \bullet (\jeden_s \bullet {}_g \jeden ) \bullet [[F]] \bullet \jeden_t \\
& = {}_p \jeden \bullet [[E]] \bullet ( {}_{\widetilde g} \jeden \bullet \jeden_{\widetilde s}) \bullet [[F]] \bullet \jeden_t \\
& = {}_p \jeden \bullet ([[E]] \bullet  {}_{\widetilde g} \jeden) \bullet (\jeden_{\widetilde s} \bullet [[F]]) \bullet \jeden_t \\
& = {}_p \jeden \bullet ({}_{\widetilde g} \jeden \bullet [[(\widetilde g)^*E]]) \bullet ([[(\widetilde s)^*F]] \bullet \jeden_{\widetilde s} ) \bullet \jeden_t \\
& = ({}_p \jeden \bullet {}_{\widetilde g} \jeden) \bullet ([[(\widetilde g)^*E]] \bullet [[(\widetilde s)^*F]]) \bullet (\jeden_{\widetilde s}  \bullet \jeden_t) \\
& = {}_{p \circ \widetilde g} \jeden \bullet [[(\widetilde g)^*E \oplus (\widetilde s)^*F]] \bullet \jeden_{t \circ \widetilde s} 
\end{align*}
\end{rem}
The above  proposition follows directly from the definitions. We also note that all the properties except for ($A_1$) also follow from the property ($A_1$), i.e., the associativity of product $\bullet$, and the formulas listed in the above Remark \ref{rem-push-pull} and Remark \ref{special cobordism}. Indeed, for example, 
\begin{itemize}
\item ($A_2$) (a) follows from (\ref{f_*}), ($A_1$) and (\ref{p_1}).
\item ($A_2$) (b) follows from (\ref{_*}), ($A_1$) and (\ref{1_p}).
\item ($A_{23}$) (c) follows from (\ref{f_*}), (\ref{f^*}), ($A_1$) and (\ref{s-g}).
\item ($A_{23}$) (d) follows from (\ref{_*}), (\ref{*g}), ($A_1$) and (\ref{s-g}).
\end{itemize}
%%%%%%%%%%%%%%%%%%%%%%%%%%%%%%%%%%
The following fact (which also follows from the above (\ref{s-g})) is emphasized for a later use. 
\begin{pro}(Pushforward-Product Property for Units (abbr. PPPU))\label{pppu} For the following fiber square
$$\xymatrix{
&& V\times_Y W\ar [dl]_{\widetilde{p}} \ar[dr]^{\widetilde{s}} &&\\
& V \ar [dr]_{s} && W \ar [dl]^{p} \\
 & &  Y && }
$$
with $s:V \to Y$ smooth and $p:W \to Y$ proper, we have
$$(\jeden_V \,{}_*s) \bullet p_*\jeden_W = (\widetilde p_* \jeden_{V\times_Y W})\, {}_*\widetilde s =  \widetilde p_* ( \jeden_{V\times_Y W} \, {}_*\widetilde s ) \in \Cal  M^{*,0}(V,W)^+.$$
\end{pro}
%%%%%
\begin{rem} We note that for the free ablian group $\Cal M^{*,\sharp}(X,Y)$ Definitions \ref{prod} and \ref{push-pull} are also well-defined and Propositions \ref{proposition1} and \ref{pppu} hold as well.
\end{rem}
%%%%%%%
%%%%%%
Here we recall that
$$\Cal M^{i, r} (X \xrightarrow f Y)^+$$
is generated by the isomorphism classes of cobordism cycles of the form 
$$[V \xrightarrow p X;E]$$
such that $p:V \to X$ is \emph{proper} and the composite $f \circ p:V \to Y$ is \emph{smooth} \emph{of relative dimension $-i$} and $r=\op{rank}(E)$. We have the ``forgetting the morphism $f$":
$$\frak F: \Cal M^{i, r} (X \xrightarrow f Y)^+ \to \Cal M^{i, r} (X,Y)^+.$$

%%%%%%%%%%
\begin{lem}\label{forget-lemma} 
The forgetting map $\frak F: \Cal M^{i, r} (X \xrightarrow f Y)^+ \to \Cal M^{i, r} (X,Y)^+$ commutes with product $\bullet$ and the pushforward and   satisfies the following simple formulas for the canonical orientation and the element $[[E]]$ for a vector bundle:
\begin{enumerate}
\item $\frak F (\alp \bullet \beta) = \frak F(\alp) \bullet \frak F(\beta).$
\item $\frak F \circ f_* = f_* \circ \frak F$. 
\item For a smooth morphism $s:X \to Y$ we have
\begin{enumerate}
\item $\frak F(\theta (s)) = \jeden_s$, i.e., 
$\frak F([X \xrightarrow {\op{id}_X} X]) =[X \xleftarrow {\op{id}_X} X \xrightarrow {s} Y]$,
where $\theta(s)=[X \xrightarrow {\op{id}_X} X]$ is the canonical orientation (see Definition \ref{canonical} and Theorem \ref{ubt} (2)).
\item $\frak F \circ s^! = s^* \circ \frak F$, where $s^!=\theta(s) \bullet$ is the Gysin homomorphism.
\end{enumerate}
\item For a vector bundle $E$ over $X$, we let $[E]:=[X \xrightarrow {\op{id}_X} X; E]$. Then for the forgetting morphism  
$\frak F: \Cal M^{i, r} (X \xrightarrow {\op{id}_X} X)^+ \to \Cal M^{i, r} (X,X)^+$ we have
$$\frak F([E]) =[[E]].$$
\end{enumerate}
\end{lem}
%%%%%%%%%
\begin{proof}(3) (a) and (4) are obvious. For the sake of convenience, we write down proofs for (1), (2) and (3)(b).
\begin{enumerate}
\item  For $[V \xrightarrow h X;E] \in \mathcal M^{*,\sharp}(X \xrightarrow f Y)^+$ and $[W \xrightarrow k Y;F] \in \mathcal M^{*,\sharp}(Y \xrightarrow g Z)^+$, we have
\begin{align*}
\frak F ([V \xrightarrow h X;E] \bullet [W \xrightarrow k Y;F] ) & = \frak F([V'  \xrightarrow{h \circ {k}''}  X; {{k}''}^*E \oplus (f' \circ {h}')^*F ]\\
& = [X \xleftarrow {h \circ {k}''} V' \xrightarrow{(g \circ f) \circ (h \circ {k}'')} Z; {{k}''}^*E \oplus (f' \circ {h}')^*F ]\\
& = [X \xleftarrow {h \circ {k}''} V' \xrightarrow{(g \circ k) \circ (f' \circ f') } Z; {{k}''}^*E \oplus (f' \circ {h}')^*F ]
\end{align*}
Here we use the following fiber squares
$$\CD
V' @> {h'} >> X' @> {f'} >> W \\
@V {{k}''}VV @V {{k}'}VV @V {k}VV\\
V@>> {h} > X @>> {f} > Y @>> {g} > Z .\endCD
$$
On the other hand, we have
\begin{align*}
\frak F ([V \xrightarrow h X;E]) \bullet \frak F([W \xrightarrow k Y;F] ) & = 
[X \xleftarrow h V \xrightarrow {f \circ h} Y; E] \bullet [Y \xleftarrow k W \xrightarrow {g \circ k} Z; F]\\
& = [X \xleftarrow {h \circ {k}''} V' \xrightarrow{(g \circ k) \circ (f' \circ f') } Z; {{k}''}^*E \oplus (f' \circ {h}')^*F ]
\end{align*}
Here we use the following diagram (reusing the above diagram):
$$\xymatrix{
& & V'\ar [dl]_{k''} \ar[dr]^{f' \circ h'} &  &\\
&  V \ar [dl]_{h} \ar [dr]^{f \circ h} && W \ar [dl]_{k} \ar[dr]^{g \circ k}\\
X & &  Y && Z. }
$$
Therefore we do have $\frak F([V \xrightarrow h X;E] \bullet [W \xrightarrow k Y;F] ) = \frak F ([V \xrightarrow h X;E]) \bullet \frak F([W \xrightarrow k Y;F] ).$
%%%%%%%%
\item For $[V \xrightarrow h X;E] \in \mathcal M^{*,\sharp}(X \xrightarrow {g \circ f} Z)^+$, we have
\begin{align*}
\frak F (f_*[V \xrightarrow  h X;E] )) &= \frak F([V \xrightarrow {f \circ h} Y;E])\\
& = [Y \xleftarrow {f \circ h} V \xrightarrow {g \circ f \circ h} Z;E]\\
& = f_*[X \xleftarrow h V \xrightarrow {g \circ f \circ h} Z;E]\\
& = f_*(\frak F([V \xrightarrow h X;E]))
\end{align*}
\item (b) 
\begin{align*}
\frak F (s^!(\alp)) & = \frak F(\theta (s) \bullet \alp) \\
& = \frak F(\theta (s)) \bullet \frak F(\alp) \quad \text {(by (1)}\\
& = \jeden_s \bullet \frak F(\alp) \quad \text {(by (3)(a))}\\
& = s^*( \frak F(\alp)).
\end{align*}
Hence we have  $\frak F \circ s^! = s^* \circ \frak F$.
\end{enumerate}
\end{proof}
%%%%%%%%%%%%
\begin{rem}
We emphasize that the forgetting morphism $\frak F $ is \emph{not} compatible with pullback. One simple reason is that ``pullback" of the bivariant theory $\Cal  M^{*,\sharp}(- \to -)^+$ and four kinds of ``pullback" of the present bi-variant theory $\Cal M^{*,\sharp}(-,-)^+$ are quite different. Thus $\frak F (g^*s)$ cannot be expressed as $g^* \frak F (s)$. This is a crucial problem and how to circumvent this problem is an issue later.
\end{rem}
%%%%%%%%
\begin{rem} Lemma \ref{forget-lemma} (3) and (4) are obvious, but will play roles as a trick later.
\end{rem}
%%%%%%%%%%
%%%%%%%%%%%%%%%%%%
\subsection{$\mathcal C$--$\mathcal S$ - correspondences}
The cobordism bicycle of the zero bundle $[X \xleftarrow p V \xrightarrow s Y; {\bf 0}]$ is the same as the isomorphism class 
$[X \xleftarrow p V \xrightarrow s Y]$ of a correspondence with a proper morphism $p:V \to X$ and a smooth morphism $s:V \to Y$. So this shall be called \emph{a proper-smooth correspondence.} Then the above Grothendieck group $\Cal M^*(X, Y)^+ = \Cal M^{*,0}(X, Y)^+$ with the second degree $r=0$ is nothing but the Grothendieck group of proper-smooth correspondences.

In general, for the class $\mathcal C$ of \emph{confined} morphisms and the class $\mathcal S$ of \emph{specialized} morphisms, the isomorphism class 
$[X \xleftarrow p V \xrightarrow s Y]$ of a correspondence with $p \in \mathcal C$ and $s \in \mathcal S$ shall be called a $\mathcal C$--$\mathcal S$-correspondence. Then the free abelian group generated by $\mathcal C$--$\mathcal S$-correspondences is denoted by $\mathbb M^{\mathcal C}_{\mathcal S}(X,Y)$. If we assume the conditions as in Remark \ref{rem-gro}, then the Grothendieck group generated by $\mathcal C$--$\mathcal S$-correspondences is denoted by $\mathbb M^{\mathcal C}_{\mathcal S}(X,Y)^+$. Then for both $\mathbb M^{\mathcal C}_{\mathcal S}(X,Y)$ and $\mathbb M^{\mathcal C}_{\mathcal S}(X,Y)^+$ Definitions \ref{prod} and \ref{push-pull} with
the data of vector bundles deleted are also well-defined and Propositions \ref{proposition1} and \ref{pppu} with 
the data of vector bundles deleted hold as well.
%%%%%%%%%%%%%%
\begin{rem}
As we know, correspondences make a category, usually called ``a category of correspondences", i.e., a morphism from $X$ to $Y$ is a correspondence $X \xleftarrow f V \xrightarrow g Y$ and the composition of two morphisms $X \xleftarrow f V \xrightarrow g Y$  and $Y \xleftarrow h  W \xrightarrow k  Z$ is nothing but the composition of these correspondences. In \cite{Yokura-enriched} $\mathbb M^{\mathcal C}_{\mathcal S}(X,Y)$ and $\mathbb M^{\mathcal C}_{\mathcal S}(X,Y)^+$ are both called \emph{an enriched category of correspondences}, because $\mathbb M^{\mathcal C}_{\mathcal S}(X,Y)$ and $\mathbb M^{\mathcal C}_{\mathcal S}(X,Y)^+$ are \emph{abelian groups}, whereas in the usual category of correspondences $hom (X,Y)$ is \emph{a set}. 
$\mathbb M^{\mathcal C}_{\mathcal S}(X,Y)$ and $\mathbb M^{\mathcal C}_{\mathcal S}(X,Y)^+$ can be generalized as follows: Let $\frak B(X,Y)$ be  abelian groups for pairs $(X, Y)$ which satisfy the following:
\begin{enumerate}
\item they are equipped with a product $\bullet': \frak B(X,Y) \otimes \frak B(Y,Z) \to \frak B(X,Z)$ which is associative with the unit $1=1_X \in \frak B(X,X)$,
\item  For confined maps $f:Y \to X$ there exists distinguished elements ${}_f 1 \in \mathcal B(X,Y)$ such that for two confined maps $f:Y \to X$ and $g: Y \to Z$ we have ${}_f 1 \bullet' {}_p 1= {}_{f \circ p} 1$.
\item  For specialized maps $s:X \to Y$ there exist distinguished elements $1_s \in \frak B(X,Y)$ such that for two specialized maps $s:X \to Y$ and $g: Y \to Z$ we have $1_s \bullet' 1_g = 1_{g \circ s}$.
\item ${}_{id_X}1 =1 = 1_{id_X}$ is the unit for the identity map $id_X:X \to X$,
\item $1_s \bullet' {}_g 1 = {}_{\widetilde g} 1 \bullet' 1_{\widetilde s}$ for the fiber square
$$
\xymatrix
{
\widetilde X \ar[r]^{\widetilde g} \ar[d]_{\widetilde s} &  X'' \ar[d]^s\\
X' \ar[r]_g & X
}
$$
with $s$ (hence $\widetilde s$ as well) being specialized  and $g$ (hence $\widetilde g$ as well) being confined.
\end{enumerate}
Then as in Remark \ref{rem-push-pull}, we can define the following two pushforwards and two pullbacks:
$$\text{$g_*(-) := {}_g 1 \bullet' (-),  {}^*g(-) := (-) \bullet' {}_g 1 $ for confined maps $g$,}$$
$$\text{$s^*(-) := 1_s \bullet' (-),  {}_*s(-) := (-) \bullet' 1_s $ for specialized maps $s$.}$$
Then we can see that 
$\frak B(X,Y)$ gives rise to a theory satisfying those properties in the above Proposition \ref{proposition1}. 
Furthermore, we suppose that 
\begin{enumerate}
\item there is a well-defined element $[[E]]'$ for each vector bundle over a space and 
\item an abelian group $\mathcal B(X,Y)$ is generated by elements of the forms ${}_p 1 \bullet' [[E]]' \bullet' 1_s$ for confined maps $g$ and specialized maps $s$ and 
\item they  satisfy the following properties
\begin{equation}
1_f \bullet' [[E]]' =[[f^*E]] \bullet' 1_f.
\end{equation}
\begin{equation}
[[E]]' \bullet' {}_g 1= {}_g 1 \bullet' [[g^*E]]'.
\end{equation}
\end{enumerate}
Then ${}_p 1 \bullet' 1_s$ and  ${}_p 1 \bullet' [[E]]' \bullet' 1_s$ shall be respectively called \emph{an ``abstract" correspondence} and \emph{an ``abstract" cobordism bicycle of vector bundle}. Then $\frak B(X,Y)$ is respecitvely \emph{a more enriched category of abstract correspondences} and \emph{a more enriched category of abstract cobordism bicycles of vector bundles}, compared with the geometric cobordism bicycle of vector bundle $[X \xleftarrow p V \xrightarrow s Y; E]= {}_p \jeden \bullet [[E]] \bullet \jeden_s$. For such a theory we do have a canonical Grothendieck transformation, e.g., 
$$ \ga: \mathbb M^{\mathcal C}_{\mathcal S}(X,Y) \to \frak B(X,Y)$$
defined by
$$\ga({}_p \jeden \bullet [[E]] \bullet \jeden_s):= {}_p 1 \bullet' [[E]]' \bullet' 1_s.$$
Which, in particular, implies that $\ga({}_p \jeden)= {}_p 1$ by considering the special case when $ [[E]] =[[{\bf 0}]]=\jeden_V$ and $\jeden_s=\jeden_V$ for the identity map $s=\op{id}_V$, similarly $\ga([[E]])= [[E]]'$ and $\ga(\jeden_s)=1_s$ by considering the corresponding special cases. Here we emphasize that a theory satisfying those properties in the above Proposition \ref{proposition1} \emph{does not necessarily} come from a certain more enriched category $\frak B(X,Y)$ of abstract correspondences or abstract cobordism bicycles of vector bundles.
\end{rem}
%%%%%%%%%%%%%%%%

We also have the following ``forgetting the map $f$" defined by $\frak F([V \xrightarrow p X]):=[X \xleftarrow p V \xrightarrow {f \circ p} Y]$:
$$\frak F: \mathbb M^{\mathcal C}_{\mathcal S}(X \xrightarrow f  Y) \to \mathbb M^{\mathcal C}_{\mathcal S}(X,Y),$$
$$\frak F: \mathbb M^{\mathcal C}_{\mathcal S}(X \xrightarrow f  Y)^+ \to \mathbb M^{\mathcal C}_{\mathcal S}(X,Y)^+,$$
which are both embeddings or monomorphisms. 

%%%%%%%%%%%%%%%%%%%%%%%%%%%%%%
\section{A bi-variant algebraic cobordism with bundles $\Omega^{*.\sharp}(X,Y)$}\label{main}

In this section we consider the ``correspondence" version of Annala's bivariant derived algebraic cobordism $\Omega^*(X \to Y)$ \cite{An} and furthermore that of Annala--Yokura's bivariant algebraic cobordism with vector bundles $\Omega^{*,\sharp}(X \to Y)$ \cite{AY} (also see \cite{An2, An3}). 
%%%%%%%%%%%%%

From now on, as in \cite{An} and \cite{AY}, we work in derived algebraic geometry, thus we work in the context of quasi-projective derived schemes over a base field of characteristic zero. These are not categories in the classical sense, but $\infty$-categories, so that one needs some small modifications. One can work either directly with an $\infty$-category with homotopy fiber squares (instead of fiber squares), and equivalence classes of arrows (instead of isomorphism classes) and a final object unique up to equivalence, e.g., as in Annala's thesis \cite{An-Th}. Or one works in the underlying homotopy category, which is a usual category so that equivalence classes map to isomorphism classes and final objects map to a usual final object unique up to isomorphism. And instead of fiber products one works with the commutative diagrams induced from the homotopy fiber squares, e.g., as in \cite{An}. So one has to be aware that these are usually not fiber squares in the underlying homotopy category as used before. For simplicity we work in the underlying homotopy category. Then the notion of a ``homotopy (or derived) fiber square” in the homotopy category is preserved by isomorphisms of commutative squares, so that all the constructions for universal bivariant theories, or those defined via 
correspondences, work as before, e.g., any two arrows 
$$
\xymatrix{
{} & Y \ar[d]\\
X \ar[r] & Z
}
$$ 
can be completed to a ``homotopy (or derived) fiber square” unique up to isomorphism. Also in this context one can speak of vector bundles $E$ (of constant rank), proper as the confined morphisms and quasi-smooth (of constant (virtual) fiber dimension) as the specialized morphisms (as discussed in more detail in \cite{AY}).
%%%%%%%%%%%%%
\subsection{A slight modification of $\Cal M^{*,\sharp}(X, Y)^{+}$}

In order to construct our bi-variant algebraic cobordism $\Omega^{*,\sharp}(X,Y)$, we modify our previous abelian groups $\Cal M^{*,\sharp}(X, Y)^{+}$ as follows.

\begin{defn} Let $X \xleftarrow p V \xrightarrow s Y$ be a correspondence such that $p:V \to X$ is \emph{proper} and $s: V \to Y$ is \emph{quasi-smooth}, and let $E$ be a complex vector bundle over $V$. Then
$(X \xleftarrow p V \xrightarrow s Y; E)$
is still called a \emph{cobordism bicycle of a vector bundle}.
\end{defn}
%%%%%%%%%%

Then we define $\Cal M^{i,r}(X, Y)^{+}$ (using the same symbol as before) to be an abelian group 
generated by the isomorphism classes of cobordism bicycles of the form $[X \xleftarrow p V \xrightarrow s Y;E]$
such that $s:V \to Y$ is a quasi-smooth morphism \emph{of virtual relative dimension $-i$} and $r=\op{rank}(E)$. 

Here we note that in the case of the bivariant algebraic cobordism with vector bundles $\Omega^{*,\sharp}(X \to Y)$ we consider
$$\mathcal M^{i,r}(X \to Y )^{+}:=\mathbb L \otimes \Cal M^{i,r}(X \to Y)^{+}$$
where $\Cal M^{i,r}(X \xrightarrow f Y)^{+}$ is generated by the isomorphism classes of cobordism cycles of the form 
$$[V \xrightarrow p X;E]$$
such that $p:V \to X$ is proper and the composite $f \circ p:V \to Y$ is quasi-smooth \emph{of virtual relative dimension $-i$} and $r=\op{rank}(E)$. 
%%%%%%%%%
\begin{rem} If we consider the case when $E$ is the zero bundle ${\bf 0}$, i.e., $[V \xrightarrow p X; {\bf 0} ]$, or equivalently we consider $[V \xrightarrow p X]$, then we have $\Cal M^{i,0}(X \xrightarrow f Y)^{+}$, which is nothing but $\Cal M^i(X \xrightarrow f Y)^{+}$ used in \cite{An}. If we use the notation in Remark \ref{rem-gro}, 
$\Cal M^i(X \xrightarrow f Y)^{+}$ is denoted by $(\mathbb M^{prop}_{qusm})^i(X \xrightarrow f Y)^+$ where $prop=\mathcal C$ is the class of proper morphisms and $qusm=\mathcal S$ is the class of quasi-smooth morphisms.
\end{rem}
%%%%%%%%%
Therefore we have the canonical morphism, ``forgetting" the morphism $f:X \to Y$:
$$\frak  F: \Cal M^{i,r}(X \xrightarrow f Y )^{+} \to \Cal M^{i,r}(X,Y )^{+}$$
defined by
$$\frak F([V \xrightarrow p X;E]):= [X \xleftarrow p V \xrightarrow {f \circ p} Y;E].$$

Then we get the same results for 
$\Cal M^{i,r}(-,-)^{+}$
 as Proposition \ref{proposition1} and Proposition  \ref{pppu} as in \S 3, so omitted for the sake of simplicity. 
%%%%%%%%%%%%

We recall the following definition of the top Chern class of a vector bundle \cite[Definition 5.5]{AY}:
%%%%%%%%%%%%%%%%
\begin{defn}(\cite[Definition 5.5]{AY}) \label{TopChernClass}
Given a vector bundle $E$ of rank $n$ on $X$, we define its \emph{top Chern class} as
\begin{equation*}
c_n(E) = {s}^*{s}_!(1_X) \in \Cal {M}^{n,0}(X \xrightarrow {\op{id}_X} X)^+ 
\end{equation*}
where $s: X \to E$ is the zero section. 
\end{defn}
%%%%%%%%%%%
We note (see \cite[Remark 5.6]{AY}) that the above $c_n(E)$ can be expressed explicitly as follows:
\begin{equation}\label{chern-explicit}
c_n(E) = [V(s) \xrightarrow {i_{V(s)}} X; {\bf 0}]
\end{equation}
where $V(s)$ is the derived vanishing locus of the zero section and $i_{V(s)}$ is the inclusion. 
Hence, it is quite natural to define the ``correspondence" version of the Chern class, still called the top Chern class of the vector bundle $E$ and denoted by $c_n(E) \in 
\Cal {M}^{n,0}(X, X)^+$
, as follows:
%%%%%%%%%%
\begin{defn}
$$c_n(E) := [X \xleftarrow {i_{V(s)}} V(s) \xrightarrow {i_{V(s)}} X;{\bf 0}] \in 
Cal M^{n,0}(X, X)^+.$$
\end{defn}
%%%%%%%%%%
%%%%%%%%
\begin{lem}\label{fcfE} Let $E$ be a vector bundle of rank $n$ over $X$.
\begin{enumerate}
\item \label{fcfl} For a proper morphism $f:X' \to X$, $f_*c_n(f^*E)=c_n(E){}^*f$, i.e.,  ${}_f \jeden\bullet c_n(f^*E)=c_n(E)\bullet {}_f \jeden.$.
\item \label{cflf} For a quasi-smooth morphism $f:X' \to X$, $c_n(f^*E){}_*f = f^*c_n(E)$, i.e.,  $c_n(f^*E) \bullet {}_f \jeden ={}_f \jeden \bullet c_n(E).$

\end{enumerate}
\end{lem}
\begin{proof} We prove only the first one, since that of the second one is similar. By definition we have
$c_n(E) = [X \xleftarrow {i_{V(s)}} V(s) \xrightarrow {i_{V(s)}} X]$, where we drop the zero bundle $\bf 0$ for the sake of simplicity. Hence, using the following diagram in which the right diamond is a fiber square:
$$\xymatrix{
& V(s)   \ar [dl]_{i_{V(s)}} \ar [dr]_{i_{V(s)}} & \widehat {V(s)} \ar[l]_{\widehat f} \ar[dr]^{\widehat {i_{V(s)}}  }\\
X &  & X & X' \ar[l]^f}
$$
we have
\begin{align*}
c_n(E){}^*f & = [X \, \, \xleftarrow {i_{V(s)} \circ \widehat f} \, \,  \widehat {V(s)}  \, \, \xrightarrow {\widehat {i_{V(s)}}} \, \, X']\\
& = [X \, \, \xleftarrow {f \circ \widehat {i_{V(s)}}} \, \,  \widehat {V(s)}  \, \, \xrightarrow {\widehat {i_{V(s)}}} \, \, X']\\
& = f_*[X' \, \, \xleftarrow {\widehat {i_{V(s)}}} \, \,  \widehat {V(s)}  \, \, \xrightarrow {\widehat {i_{V(s)}}} \, \, X']
\end{align*}
Since $\widehat {V(s)}$ is the pullback of the vanishing locus $V(s)$ of the zero section $s$ by the morphism $f:X \to X'$, $\widehat {V(s)}$ is equal to the vanishing locus of the pullbacked zero section $f^*s$ of the pullbacked vector bundle $f^*E$, i.e., we have
$c_n(f^*E)=[X' \, \, \xleftarrow {\widehat {i_{V(s)}}} \, \,  \widehat {V(s)}  \, \, \xrightarrow {\widehat {i_{V(s)}}} \, \, X']$. Hence we have $f_*c_n(f^*E)=c_n(E){}^*f$.
\end{proof}
%%%%%%%%%%%%%%%%%%
%%%%%%%%%%%%%%
For a vector bundle  $L$ of rank $\ell$ over $X$ or a vector bundle $M$ of rank {m} over $Y$, the homomorphisms
$$c_{\ell}(L) \bullet: \Cal M^{i,r}(X,Y)^+ \to \Cal  M^{i+\ell,r}(X,Y)^+, \quad \bullet \, c_m(M): \Cal M^{i,r}(X,Y)^+ \to \Cal M^{i+m,r}(X,Y)^+$$
defined respectively by $c_{\ell}(L) \bullet \alp$ and $\beta \bullet c_m(M)$,  are called the ``Chern class operators". 
%%%%%%%%%%%%%%%
\begin{pro}\label{ch-op} The Chern class operators of line bundles satisfy the following properties.
\begin{enumerate}
\item (identity): If $L$ and $L'$ are line bundles over $X$ and isomorphic and if $M$ and $M'$ are line bundles over $Y$ and isomorphic, then 
$$c_1(L) \bullet  = c_1(L') \bullet:\Cal  M^{i,r}(X,Y)^+ \to \Cal M^{i+1,r}(X,Y)^+,$$
$$ \bullet \, c_1(M) =\bullet \, c_1(M') : \Cal  M^{i,r}(X,Y)^+ \to \Cal M^{i+1,r}(X,Y)^+.$$
%%%%%%%%%%
\item (commutativity): If $L$ and $L'$ are line bundles over $X$ and if $M$ and $M'$ are line bundles over $Y$, then 
$$c_1(L) \bullet c_1(L') \bullet  = c_1(L') \bullet c_1(L) \bullet  :\Cal  M^{i,r}(X,Y)^+ \to \Cal  M^{i+2,r}(X,Y)^+,$$
$$\bullet c_1(M)\bullet c_1(M')= \bullet c_1(M')\bullet c_1(M) : \Cal  M^{i,r}(X,Y)^+ \to \Cal  M^{i+2,r}(X,Y)^+.$$
%%%%%%%%%%%
\item (compatibility with product) Let $L$ be a line bundle over $X$ and $N$ be a line bundle over $Z$. For $\alp \in \Cal  M^{i,r}(X,Y)^+$ and $\be \in \Cal  M^{j,k}(Y,Z)^+$, then 
 $$ c_1(L) \bullet (\alp \bullet \be) = \Bigl (c_1(L) \bullet \alp \Bigr ) \bullet \be, \quad (\alp \bullet \be) \bullet c_1(N) = \alp \bullet \Bigl (\be \bullet c_1(N) \Bigr ).$$
%%%%%%%%%%%%
\item (compatibility with pushforward = ``projection formula") 
For a proper morphism $f:X \to X'$ and a line bundle $L$ over $X'$ and for a quasi-smooth morphism $g:Y \to Y'$ and a line bundle $M$ over $Y'$ we have that for $\alp \in \Cal  M^{i,r}(X,Y)^+$
$$ f_*\bigl (c_1(f^*L) \bullet \alp \bigr) = c_1(L) \bullet f_*\alp, \,  \bigl  (\alp \bullet c_1(g^*M) \bigr) \, {}_*g = \alp \, {}_*g \bullet  c_1(M).$$
%%%%%%%%%%%%
\item (compatibility with pullback =``pullback formula") 
For a quasi-smooth morphism $f:X' \to X$ and a line bundle $L$ over $X$ and for a proper morphism $g:Y' \to Y$ and a line bundle $M$ over $Y$ we have that for $\alp \in \Cal  M^{i,r}(X,Y)^+$
$$ f^*\bigl (c_1(L) \bullet \alp \bigr) = c_1(f^*L) \bullet f^*\alp, \, \bigl (\alp \bullet c_1(M) \bigr) \, {}^*g= \alp \, {}^*g \bullet c_1(g^*M).$$
%%%%%%%%%%
\item (Pullback Property for Unit (abbr. PPU))\label{ppu}  For a quasi-smooth morphism $f:X' \to X$ and a line bundle $L$ over $X$ and for a proper morphism $g:Y' \to Y$ and a line bundle $M$ over $Y$ we have that 
for $\jeden_X \in \Cal  M^{0,0}(X,X)^+$ and $\jeden_Y \in \Cal  M^{0,0}(Y,Y)^+$
$$c_1(f^*L) \bullet f^*\jeden_X = f^*\jeden_X \bullet c_1(L) \in \Cal  M^{\op{dim}f+1, 0}(X',X)^+,$$
$$\jeden_Y \, {}^*g \bullet c_1(g^*M) = c_1(M) \bullet \jeden_Y \, {}^*g \in \Cal  M^{1,0}(Y,Y')^+. \qquad $$

In particular
\item (Commutativity of the unit and Chern class) \label{u-c}
Let $L$ be a line bundle over $X$. Then we have
$$c_1(L) \bullet \jeden_X = \jeden_X \bullet c_1(L).$$ 
\end{enumerate}
\end{pro} 
%%%%%%%%%%%%%%
%%%%%%%%%%%%%
\begin{proof} It suffices to show (2), since (1), (3) and (7) are clear and (4), (5) and (6) follow from the above Lemma \ref{fcfE} and the associativity of product $\bullet$.

We just show the first one.  
 Let $s:X \to L$ and $s':X \to L'$ be the zero sections.
 It follows from the fact that the derived fiber product $V(s) \times _X V(s') = V((s,s'))$ where
 $(s,s'):X \to L \oplus L'$ is the zero section. Hence by the definition of the product $\bullet$ we have
 \begin{align*}
 c_1(L) \bullet c_1(L') & = [X \xleftarrow {i_{V(s)}} V(s) \xrightarrow {i_{V(s)}} X;{\bf 0}] \bullet [X \xleftarrow {i_{V(s')}} V(s') \xrightarrow {i_{V(s')}} X;{\bf 0}] \\
 & = [X \xleftarrow {i_{V((s,s'))}} V((s,s')) \xrightarrow {i_{V((s,s'))}} X;{\bf 0}] \\
 &= c_2(L \oplus L') = c_2(L' \oplus L) =  c_1(L') \bullet c_1(L).
 \end{align*}
\end{proof}
%%%%%%%%%%%%
%%%%%%%%
Here we recall the definition of \emph{a bi-variant theory} \cite{Yo-NYJM}, which is defined as one similar to the definition of Fulton--MacPherson's bivariant theory. In this sense, it could be called ``a bi-variant theory of Fulton--MacPherson-type ". 
\begin{defn}[Bi-variant theory]\label{bi-va-defn} An association $\Cal   B$ assigning to a pair $(X,Y)$ a graded abelian group $\Cal   B^*(X,Y)$ is called a \emph{bi-variant theory}
provided that

\noindent
(1) it is equipped with the following three operations

\begin{enumerate}
\item (Product)  \quad $\bullet: \Cal   B^i(X, Y) \times \Cal   B^j(Y,Z) \to \Cal B^{i+j}(X, Z)$

\item (Pushforward) 
\begin{enumerate}
\item For a \emph{proper} morphism $f:X \to X'$, $f_*:\Cal   B^i(X,Y) \to \Cal   B^i(X',Y)$.
\item For a \emph{smooth} morphism $g:Y \to Y'$, ${}_*g: \Cal   B^i(X,Y) \to \Cal   B^{i-\op{dim}g}(X,Y')$.
\end{enumerate}

\item (Pullback) 
\begin{enumerate}
\item For a \emph{smooth} morphism $f:X' \to X$, $f^*: \Cal   B^i(X,Y) \to  \Cal   B^{i-\op{dim}f}(X',Y)$.
\item For a \emph{proper} morphism $g:Y' \to Y$, ${}^*g: \Cal   B^i(X,Y) \to  \Cal   B^i(X,Y')$.
\end{enumerate}
\end{enumerate}
%%%%%%%%%%%%%%%%%%

\noindent
(2) the three operations satisfy the following nine properties as in Proposition \ref{proposition1}:
\begin{enumerate}
\item[($A_1$)] Product is associative.
\item[($A_2$)] Pushforward is functorial. ((a), (b))
\item[($A_2$)'] Proper pushforward and smooth pushforward commute.
\item[($A_3$)] Pullback is functorial.  ((a), (b))
\item[($A_3$)'] Proper pullback and smooth pullback commute.
\item[($A_{12}$)] Product and pushforward commute.  ((a), (b))
\item[($A_{13}$)] Product and pullback commute.  ((a), (b))
\item[($A_{23}$)] Pushforward and pullback commute.  ((a), (b), (c), (d))
\item[($A_{123}$)] Projection formula.  ((a), (b))
\end{enumerate}
%%%%%%%%%

\noindent
(3) $\Cal  B$ has units, i.e., there is an element $1_X \in \Cal  B^0(X,X)$ such that $1_X \bullet \alp = \alp$ for any element $\alp \in \Cal   B(X, Y)$ and $\be \bullet 1_X = \be$ for any element $\be \in \Cal   B(Y, X)$.
%%%%%%%%%%%

\noindent
(4) $\Cal  B$ satisfies PPPU (as in Proposition \ref{pppu}).
%%%%%%%%%%%

\noindent
(5) $\Cal B$ is equipped with the Chern class operators satisfying the properties in Proposition \ref{ch-op}.
\end{defn}

%%%%%%%%%%%%
\begin{cor} $\Cal  M^{*,\sharp}(X, Y)^+$ is a bi-variant theory (with respect the grading $*$, or ignoring the grading $\sharp$.).
\end{cor}

%%%%%%%%%%
\subsection{Grothendieck transformation and a bi-variant ideal}
%%%%%%%%%%%%%%%
%%%%%%%%%%%
\begin{defn}
Let $\Cal   B, \Cal   B'$ be two bi-variant theories on a category $\Cal V$. A {\it Grothendieck transformation} from $\Cal   B$ to $\Cal   B'$, $\ga : \Cal   B \to \Cal   B'$
is a collection of homomorphisms
$\Cal   B(X, Y) \to \Cal   B'(X,Y)$
for a pair $(X,Y)$ in the category $\Cal V$, which preserves the above three basic operations and the Chern class operator: 
\begin{enumerate}
\item $\ga (\alp \bullet_{\Cal   B} \be) = \ga (\alp) \bullet _{\Cal   B'} \ga (\be)$, 
\item $\ga(f_{*}\alp) = f_*\ga (\alp)$ and $\ga(\alp \, {}_*g) = \ga (\alp) \, {}_*g$,
\item $\ga (g^* \alp) = g^* \ga (\alp)$ and $\ga (\alp \, {}^*f) = \ga (\alp) \, {}^*f$,
\item $\ga(c_1(L) \bullet \alp) = c_1(L) \bullet_{\Cal   B} \ga(\alp)$ and $\ga(\alp \bullet c_1(M) ) = \ga(\alp) \bullet_{\Cal   B} c_1(M)$.
\end{enumerate}
\end{defn}
%%%%%%%%%%%%%%%%%
%%%%%%%%%%%%%%%%%%
Motivated by the definition of the bivariant ideal considered in \cite{An} and \cite{AY}, in an analogous manner we can define  a ``bi-variant" ideal as follows:
%%%%%%%%%%%
Let $\Cal B$ be a bi-variant theory. A \emph{bi-variant ideal} $\Cal I  \subset \Cal B$ is defined to consist of (graded) subgroups $\Cal I(X, Y) \subset \Cal B(X, Y)$ for each pair $(X,Y)$ such that the following hold: Let $\alp \in \Cal I(X,Y)$.
\begin{enumerate}
\item (pushforward) 
  \begin{enumerate}
   \item $f_*\alp \in \Cal I(X',Y)$ for a \emph{proper} morphism $f:X \to X'$, i.e., $f_*: \Cal I(X,Y) \to \Cal I(X',Y),$
   \item $\alp {}_*g \in \Cal I(X,Y')$ for a \emph{quasi-smooth} morphism $g:Y \to Y'$, i.e., ${}_*g: \Cal I(X,Y) \to \Cal I(X,Y'),$
    \end{enumerate}

\item (pullback)
    \begin{enumerate}
    \item $f^*\alp \in \Cal I (X',Y)$ for a \emph{quasi-smooth} morphism $f:X' \to X$, i.e.,$f^*: \Cal I(X,Y) \to \Cal I(X',Y),$
    \item $\alp {}^*g \in \Cal I(X,Y')$ for a \emph{proper} morphism $g:Y \to Y'$, i.e., ${}^*g: \Cal I(X,Y) \to \Cal I(X,Y'),$
     \end{enumerate}

\item  (product)
  \begin{enumerate}
   \item $\beta \bullet \alp \in \Cal I(X',Y)$ for any $\beta \in \Cal B(X', X)$, i.e., $\bullet: \Cal B(X', X) \otimes \Cal I(X,Y) \to \Cal I(X',Y)$,
   \item $\alp \bullet \delta \in \Cal I(X,Y')$ for any $\delta \in \Cal B(Y,Y')$, i.e., $\bullet: \Cal I(X, Y) \otimes \Cal B(Y,Y') \to \Cal I(X,Y')$,
      \end{enumerate}
      
\item (Chern class operator)
  \begin{enumerate}
   \item $c_1(L) \bullet \alp \in \Cal I(X,Y)$ for any line bundle $L$ over $X$, i.e., $c_1(L) \bullet {}: \Cal I(X,Y) \to \Cal I(X,Y),$
   \item $\alp \bullet c_1(M) \in \Cal I(X,Y)$ for any line bundle $M$ over $Y$, i.e., ${} \bullet c_1(M): \Cal I(X,Y) \to \Cal I(X,Y).$
 \end{enumerate}
\end{enumerate}
%%%%%%%%%%%
\begin{rem}\label{bi-var-ideal}
We note that in the case of ``bi-variant" ideal of $\Cal  M^{*,\sharp}(X, Y)^+$, \emph{thanks to the formulas listed in Remark \ref{rem-push-pull}}, the above (1) (pushforward) and (2) (pullback) are special cases of (3) product, and it is also the case for (4) (Chern class operator). Therefore, any ``bi-variant ideal" of $\Cal  M^{*,\sharp}(X, Y)^+$ is simply defined just as (3) (product).  This simplicity will be used later.
\end{rem}
%%%%%%%%%%%%
It is easy to see the following:
%%%%%%%%%%%%
\begin{pro}\label{quotient-pro}
\noindent
\begin{enumerate}
\item The (object-wise) kernel of a Grothendieck transformation $\gamma: \Cal B\to \Cal B'$ is a bi-variant ideal, i.e., the kernel $\ga^{-1}(X,Y)$ of the homomorphism $\ga: \Cal B(X,Y) \to \Cal B'(X,Y)$ gives rise to a bi-variant ideal, denoted by $\ga^{-1}$,  of the bi-variant theory $\Cal B$.

\item Given a bi-variant ideal $\Cal I \subset \Cal B$, the following quotient $\Cal B / \Cal I$ becomes a bi-variant theory by setting 
$$(\Cal B / \Cal I)(X,Y) := \Cal B(X,Y) / \Cal I(X, Y)$$
and by taking the bi-variant operations to be the ones induced by $\Cal B$. Namely, they are defined as follows: 
\begin{enumerate}
\item (product): the product operation
$$\bullet: (\Cal B / \Cal I)^*( X,Y) \otimes (\Cal B / \Cal I)^*( Y, Z) \to
(\Cal B / \Cal I)^*( X, Z)$$
is  defined by $[\alp] \bullet [\be]:=[\alp \bullet \be].$

\item (pushforward): 
  \begin{enumerate}
   \item for a \emph{proper} morphism $f:X \to X'$, the pushforward 
   $$f_*: (\Cal B / \Cal I)(X,Y) \to (\Cal B / \Cal I)(X',Y)$$
   is defined by $f_*[\alp]:=[f_*\alp]$.
  
   \item for a \emph{quasi-smooth} morphism $g:Y \to Y'$, the pushforward
   $${}_*g: (\Cal B / \Cal I)(X,Y) \to (\Cal B / \Cal I)(X,Y')$$
   is defined by $[\alp]{}_*g:=[\alp{}_*g]$.
    \end{enumerate}

\item (pullback)
 \begin{enumerate}
   \item for a \emph{quasi-smooth} morphism $f:X' \to X$, the pullback
   $$f^*: (\Cal B / \Cal I)(X,Y) \to (\Cal B / \Cal I)(X',Y)$$
   is defined by $f^*[\alp]:=[f^*\alp]$.
  
   \item for a \emph{proper} morphism $g:Y \to Y'$, the pull
   $${}^*g: (\Cal B / \Cal I)(X,Y) \to (\Cal B / \Cal I)(X,Y')$$
   is defined by $[\alp]{}^*g:=[\alp{}^*g]$.
    \end{enumerate}
\item (Chern class operator)
  \begin{enumerate}
   \item for any line bundle $L$ over $X$ the Chern class operator
   $$c_1(L) \bullet: (\Cal B / \Cal I)(X,Y) \to (\Cal B / \Cal I)(X,Y)$$
   is defined by $c_1(L) \bullet [\alp] := [c_1(L)\bullet \alp]$. (To be more precise, $c_1(L) \bullet$ should be written as
   $[c_1(L)] \bullet$.)
   \item for any line bundle $M$ over $Y$ the Chern class operator
   $${} \bullet c_1(M): (\Cal B / \Cal I)(X,Y) \to (\Cal B / \Cal I)(X,Y)$$
   is defined by $[\alp] \bullet c_1(M):= [\alp \bullet c_1(M)]$. (As above, $\bullet c_1(M)$ should be written as
   $\bullet c_1(M)]$.)

   \end{enumerate}
\end{enumerate}
\end{enumerate}
\end{pro}
%%%%%%%%%
\begin{rem} Proposition \ref{quotient-pro} (2) means, in other words, that the quotient morphism 
$$\Theta: \Cal B \to \Cal B / \Cal I$$
 is a Grothendieck transformation of bi-variant theories.
\end{rem}
%%%%%%%%%
%%%%%%%%%%%%
\subsection{A bi-variant algebraic cobordism with bundles $\Omega^{*,\sharp}(X,Y)$}\label{bi-variant}

First we recall the definition of the bivariant algebraic cobordism with vector bundles $\Omega^{*,\sharp}(X \xrightarrow f Y)$ \cite[Definition 5.7]{AY}. First we set
$$\mathcal M^{i,r}(X,Y)^{+}:=\mathbb L \otimes \Cal M^{i,r}(X, Y)^{+}$$
where $\mathbb L$ is the Lazard ring. 

%%%%%%%%%%
\begin{defn}[The bivariant algebraic cobordism with vector bundles]
$$\Omega^{*,\sharp}(X \xrightarrow f Y):= \frac{\mathcal{M}^{*,\sharp}(X \xrightarrow f Y)^+}{\langle \mathcal R^{\op{LS}} \rangle (X \xrightarrow f Y)}$$
where $\langle \mathcal R^{\op{LS}}\rangle$ is the bivariant ideal generated by the bivariant subset $\mathcal R^{\op{LS}}$, which is defined as follows:
$$
\mathcal R^{\op{LS}}(X \to Y) := \begin{cases}
\mathcal R^{\op{LS}}(X \to pt), & \quad Y =pt,\\
\qquad \emptyset, & \quad Y \not =pt.
\end{cases}
$$
Here $\mathcal R^{\op{LS}}(X \to pt)$ is defined to be the kernel of the (surjective) morphism
$$\mathcal M^{-*,0}(X \to pt)^+=\mathbb L \otimes \Cal M^{-*,0}(X \to pt)^+\to d\Omega_*(X)=\Omega^*(X \to pt).$$
which is the $\mathbb L$-linear map, induced from a group homomorphism $\Cal M^{-*,0}(X \to pt)^+\to d\Omega_*(X)$ together with the ring homomorphism $\mathbb L \to d\Omega_*(pt)$ classifying the \emph{formal group law} for the first Chern class of the tensor product of line bundles.
\end{defn}
%%%%%%%%%%%
Here is an explicit description of the bivariant ideal $\langle \mathcal R^{\op{LS}} \rangle (X \xrightarrow f Y)$ \cite[Proposition 5.12]{AY}, which is a ``vector bundle" version of \cite[Lemma 3.8]{An}:
%%%%%%%%%%
\begin{pro}\label{bi-ideal} The bivariant ideal $\langle \mathcal R^{\op{LS}}\rangle$ satisfies that $\langle \mathcal R^{\op{LS}}\rangle (X \xrightarrow f Y)$ consists of linear combinations of elements of the form
\begin{equation}\label{bi-form}
h_* \bigl ([E] \bullet \alp_0 \bullet g^*s \bullet \beta_0 \bigr)
\end{equation}
where $h:A'' \to X, g:B' \to pt, \alp_0 \in \mathcal M^{*,0}(A'' \to A')^+, \beta_0 \in \mathcal M^{*,0}(B' \to Y)^+$ and $s \in \mathcal R^{\op{LS}}(A \to pt)$ are as in the following diagram
\[
\xymatrix{
& X \ar@/^12pt/[drr]^f \\
A'' \ar[r]^{\Maru{$\alp_0$}} \ar@/^10pt/[ru]^h & A' \ar[r]^{\Maru{$g^*s$}} \ar[d] & B' \ar[d]^g  \ar[r]^{\Maru{$\beta_0$}} & Y \\
& A  \ar[r]^{\maru{s}} & pt
} 
\]
$E$ a vector bundle on $A''$ and $h:A'' \to X$ is a proper morphism. Here note that $[E]:=[A'' \xrightarrow {\op{id}_{A''}} A''; E] \in \mathcal{M}^{*,*}(A'' \xrightarrow {\op{id}_{A''}} A'')^+$.
\end{pro}
%%%%%%%
\begin{rem}\label{bi-ideal-rem} Here we emphasize that if we let $E$ be the $0$ bundle in (\ref{bi-form}) or equivalently if we delete $[E] \bullet $ from (\ref{bi-form}), the above proposition is nothing but  \cite[Lemma 3.8]{An}.
\end{rem}
%%%%%%%%
The linear extension of the forgetting $\frak F: \Cal M^{i,r}(X \xrightarrow f Y)^{+} \to \Cal M^{i,r}(X,Y)^{+}$ by the Lazard ring $\mathbb L$ shall be denoted by the same symbol $\frak F: \mathcal M^{i,r}(X \xrightarrow f Y)^{+} \to \mathcal M^{i,r}(X,Y)^{+}$, which is $\mathbb L$-linear.

In order to define  a ``correspondence" version $\langle \mathcal R^{\op{LS}}\rangle (X,Y)$ of $\langle \mathcal R^{\op{LS}}\rangle (X \xrightarrow f Y)$, one might be tempted to simply define it as the image of $\langle \mathcal R^{\op{LS}}\rangle (X \xrightarrow f Y)$ by $\frak F: \mathcal M^{i,r}(X \xrightarrow f Y)^{+} \to \mathcal M^{i,r}(X,Y)^{+}$:
$$\frak F \left (\langle \mathcal R^{\op{LS}}\rangle (X \xrightarrow f Y) \right)$$
which consists of linear combinations of the form $\frak F\left( h_* \bigl ([E] \bullet \alp_0 \bullet g^*s \bullet \beta_0 \bigr) \right )$
by Proposition \ref{bi-ideal}. 
It follows from Lemma \ref{forget-lemma} that we have
\begin{equation}\label{forget-formula}
\frak F\left( h_* \bigl ([E] \bullet \alp_0 \bullet g^*s \bullet \beta_0 \bigr) \right ) = h_* \Bigl (\frak F ([E] ) \bullet \frak F (\alp_0 ) \bullet \frak F (g^*s)  \bullet \frak F (\beta_0)  \Bigr).
\end{equation}
%%%%%%%%%%
Now we state our theorem:
%%%%%%%
\begin{thm} There exists a bi-variant ideal $\langle \mathcal R^{\op{LS}}\rangle(-,-) \subset \mathcal{M}^{*,\sharp}(-,-)^+$ which is invariant under the $\mathbb L$-module structure such that the following hold:
\begin{enumerate}
\item $\frak F \left (\langle \mathcal R^{\op{LS}}\rangle (X \xrightarrow f Y) \right) \subset \langle \mathcal R^{\op{LS}}\rangle (X,Y)$.
\item We set
$$\Omega^{*,\sharp}(X, Y) := \frac{\mathcal M^{*,\sharp}(X,Y)^+}{\langle \mathcal R^{\op{LS}}\rangle (X,Y)}.$$
Then the forgetting homomorphism $\frak F : \mathcal M^{*,\sharp}(X \to Y)^+ \to \mathcal M^{*,\sharp}(X,Y)^+$ descends to
$$\widetilde {\frak F}: \Omega^{*,\sharp}(X \to Y) \to \Omega^{*,\sharp}(X, Y)$$
defined by $\widetilde {\frak F} ([\alp]):=[\frak F(\alp)]$, where on the left-hand side $[\alp]:=\alp + \langle \mathcal R^{\op{LS}}\rangle(X \to Y)$ and on the right-hand side $[-]$ is $[\beta]:=\beta + \langle \mathcal R^{\op{LS}}\rangle(X,Y)$, 
i.e., we have the following commutative diagram:
\begin{equation}\label{cd1}
\CD
\mathcal M^{*,\sharp}(X \xrightarrow f Y)^+  @> {\frak F} >> \mathcal M^{*,\sharp}(X, Y)^+  \\
@V {\pi} VV @VV {\pi} V\\
\Omega^{*,\sharp}(X \xrightarrow f Y)  @>> {\widetilde {\frak F}} > \Omega^{*,\sharp}(X,Y)\endCD
\end{equation}
where both $\pi$'s are the quotient maps.
\item \label{(X,*)} When $Y=pt$ is a point, we have the isomorphism:
$$\widetilde {\frak F}: \Omega^{*,\sharp}(X \to pt) \xrightarrow {\cong} \Omega^{*,\sharp}(X, pt),$$
thus we have $\Omega^{*,\sharp}(X, pt) \cong \Omega^{*,\sharp}(X \to pt) = \Omega_{-*,\sharp}(X) \cong \omega_{-*,\sharp}(X).$
\end{enumerate}
\end{thm}
%%%%%%%%%
\begin{proof}
In order to see how to define a reasonable ``correspondence" version $\langle \mathcal R^{\op{LS}}\rangle (X,Y)$, we express the right hand side of (\ref{forget-formula}) simply as follows:
\begin{align*}
h_* \Bigl (\frak F ([E] ) \bullet \frak F (\alp_0 ) \bullet \frak F (g^*s)  \bullet \frak F (\beta_0)  \Bigr) & = h_*\frak F ([E] ) \bullet \frak F (\alp_0 ) \bullet \frak F (g^*s)  \bullet \frak F (\beta_0) \quad \text{(by $A_{12}$)} \\
& = \Bigl (h_*\frak F ([E]) \bullet \frak F (\alp_0 ) \Bigr ) \bullet \frak F (g^*s)  \bullet \frak F (\beta_0).
\end{align*}
Here clearly we have
$$h_*\frak F ([E] ) \bullet \frak F (\alp_0 )  \in  \mathcal M^{*,\sharp}(X,A')^{+}, \quad \frak F (g^*s)  \in \mathcal M^{*,\sharp}(A', B)^{+}, \quad \frak F (\beta_0)   \in \mathcal M^{*, \sharp}(B,Y)^{+}.$$
Instead of giving an explicit form for an ``correspondence" version $\langle \mathcal R^{\op{LS}}\rangle (X,Y)$ just like (\ref{bi-ideal}), we define the following less explicit one: $\langle \mathcal R^{\op{LS}}\rangle (X,Y)$ is defined to be generated (as an abelian group) by elements of the form
\begin{equation}\label{b-v-ideal}
\alp_0' \bullet \frak F (g^*s)  \bullet \beta_0',
\end{equation}
where  $\alp_0'  \in  \mathcal M^{*,\sharp}(X,A')^{+}, \beta_0' \in \mathcal M^{*, \sharp}(B,Y)^{+}$, and $g$ and $s$ are as in (\ref{bi-ideal}).
Then clearly it follows from (\ref{forget-formula}) that we do have
$$\frak F \left (\langle \mathcal R^{\op{LS}}\rangle (X \xrightarrow f Y) \right) \subset \langle \mathcal R^{\op{LS}}\rangle (X,Y).$$
Furthermore it follows from (\ref{b-v-ideal}) that the following product property holds:
for $\alp \in \langle \mathcal R^{\op{LS}}\rangle(X,Y)$
 \begin{enumerate}
  \item $\beta \bullet \alp \in \langle \mathcal R^{\op{LS}}\rangle (X',Y)$ for any $\beta \in \mathcal{M}^{*,\sharp}(X', X)^+$, 
  \item $\alp \bullet \delta \in \langle \mathcal R^{\op{LS}}\rangle (X,Y')$ for any $\delta \in \mathcal{M}^{*,\sharp}(Y,Y')^+$.
\end{enumerate}
Therefore $\langle \mathcal R^{\op{LS}}\rangle(-,-)$ is a bi-variant ideal in $\mathcal{M}^{*,\sharp}(-,-)^+$. Moreover $\langle \mathcal R^{\op{LS}}\rangle (X,Y)$ is automatically invariant under the $\mathbb L$-module structure, i.e., stable under the $\mathbb L$-action coming from the ring homomorphism $\mathbb L \to \mathcal M^{*, \sharp}(Z,Z)^{+}:\ell \to \ell \otimes \jeden_Z$ for $Z =X, Y$. Clearly we obtain the above commutative diagram (\ref{cd1}).\\
%%%%%%%%%%%%%%
\indent Next we observe that the forgetting homomorphism
$$\frak F: \mathcal M^{*,\sharp}(X \to pt)^+ \to \mathcal M^{*,\sharp}(X,pt)^+$$
is an isomorphism because $\frak F([V \xrightarrow h X; E])=[X \xleftarrow h V \xrightarrow {} pt;E]$ and the inverse of $\frak F$
$$\frak F^{-1}: \mathcal M^{*,\sharp}(X,pt)^+ \to \mathcal M^{*,\sharp}(X \to pt)^+$$
 defined by $\frak F^{-1}([X \xleftarrow h V \xrightarrow {} pt;E]) := [V \xrightarrow h X; E]$ is well-defined. As observed above, 
we have
$$\frak F (\langle \mathcal R^{\op{LS}}\rangle(X \to pt)) \subset \langle \mathcal R^{\op{LS}}\rangle(X,pt).$$
Now we want to show the following reverse inclusion in the case of $X \to pt$:
\begin{equation}\label{supset}
\frak F (\langle \mathcal R^{\op{LS}}\rangle(X \to pt)) \supset \langle \mathcal R^{\op{LS}}\rangle(X,pt).
\end{equation}
First, we note that, as defined above, $\langle \mathcal R^{\op{LS}}\rangle (X,pt)$ is an abelian group generated by elements of the form
\begin{equation*}
\alp_0' \bullet \frak F (g^*s)  \bullet \beta_0' 
\end{equation*}
where $\alp_0'  \in  \mathcal M^{*,\sharp}(X,A')^{+}, \beta_0' \in \mathcal M^{*, \sharp}(B,pt)^{+}$, and $g$ and $s$ are the same as above. Then clearly we do have that $\beta'_0 = \frak F(\beta_0)$ for some $\beta_0 \in \mathcal M^{*, \sharp}(B \to pt)^{+}$, which is the tautological inverse of $\beta_0$, as observed above. Since $\frak F$ is $\mathbb L$-linear, it is enough to assume that $\alp_0'= [X \xleftarrow p A'' \xrightarrow {s'} A';E]$ (where $p$ is proper, $s'$ is quasi-smooth and $E$ is a vector bundle over $A''$), which can be expressed as
\begin{align*}
\alp_0' & = [X \xleftarrow p A'' \xrightarrow {s'} A';E] \\
& ={}_p \jeden \bullet [[E]] \bullet \jeden_{s'} \\
&= {}_p \jeden \bullet \frak F([E]) \bullet \frak F(\theta (s')) \quad \text{(by Lemma \ref{forget-lemma}, where ``smooth" is replaced by ``quasi-smooth")}\\
& = p_*(\frak F([E])) \bullet \frak F(\theta (s'))\\
&= \frak F(h_*[E]) \bullet \frak F(\theta (s')) \quad \text{(by Lemma \ref{forget-lemma} (2))}\\
& =\frak F(h_*[E] \bullet \theta (s'))  \quad \text{(by Lemma \ref{forget-lemma} (1))}\\
\end{align*}
Therefore we have
\begin{align*}
\alp_0' \bullet \frak F (g^*s)  \bullet \beta_0' & = \frak F(h_*[E] \bullet \theta (s')) \bullet \frak F (g^*s)  \bullet \frak F(\beta_0)\\
& =  \frak F \Bigl (h_*[E] \bullet \theta (s') \bullet g^*s  \bullet \beta_0 \Bigr)\\
& = \frak F \Bigl (h_* ([E] \bullet \theta (s') \bullet g^*s \bullet \beta_0) \Bigr)\\
& \in \frak F (\langle \mathcal R^{\op{LS}}\rangle(X \to pt)),
\end{align*}
which implies the above inclusion (\ref{supset}).
Hence we have $\frak F (\langle \mathcal R^{\op{LS}}\rangle(X \to pt)) = \langle \mathcal R^{\op{LS}}\rangle(X,pt).$
Therefore we see that the above isomorphism $\frak F: \mathcal M^{*,\sharp}(X \to pt)^+ \to \mathcal M^{*,\sharp}(X,pt)^+$ implies the isomorphism 
$$\Omega^{*,\sharp}(X \to pt) \cong \Omega^{*,\sharp}(X, pt) .$$
\end{proof}
}
%%%%%%
%%%%%%%
\begin{prob} Would it be possible to give some conditions on $f:X \to Y$, $X$ and $Y$ so that we have an isomorphism $\Omega^{*,\sharp}(X \xrightarrow f Y) \cong \Omega^{*,\sharp}(X,Y)$?
\end{prob}
%%%%%%%%%%%
\begin{rem} The forgetting map $\widetilde {\frak F}: \Omega^{*,\sharp}(X \to Y) \to \Omega^{*,\sharp}(X, Y)$ satisfies the following:
\begin{enumerate}
\item $\widetilde {\frak F}$ is commutative with the product $\bullet$,
\item $\widetilde {\frak F}$ is commutative with the pushforward, i.e., $\widetilde {\frak F} \circ f_* = f_* \circ \widetilde {\frak F}.$
\item $\widetilde {\frak F}$ is commutative the Chern class operators, which are special cases of (1).
\item As to the pullback, we cannot expect that the following diagram is commutative, 
\begin {equation}\label{dia}
\xymatrix
{
\Omega^{*,\sharp}(X \xrightarrow f Y) \ar[r]^{g^*} \ar[d]_{\frak F} &  \Omega^{*,\sharp}(X' \xrightarrow {f'} Y') \ar[d]_{\frak F}\\
\Omega^{*,\sharp}(X,Y)  \ar[r]_{(g')^* \circ (\, {}^*g)} & \Omega^{*,\sharp}(X',Y') 
}
\end{equation}
since the definitions of pullback are different. For the sake of simplicity, we show it. Here we consider the following fiber square 
$$\CD
V' @> g'' >> V \\
@V {h'} VV @VV {h}V\\
X' @> g' >> X \\
@V f' VV @VV f V\\
Y' @>> g > Y. \endCD
$$
for the pullback $g^*:\Omega^{*,\sharp}(X \xrightarrow f Y)  \to \Omega^{*,\sharp}(X' \xrightarrow {f'} Y')$ 
%%%%%%%%%%
\begin{align*}
(\frak F \circ g^*)([V \xrightarrow h X;E) &= \frak F ([V' \xrightarrow h' X'; (g'')^*E])\\
& =[X' \xleftarrow h' V' \xrightarrow {f' \circ h'} Y'; (g'')^*E]).
\end{align*}
%%%%%%%
\begin{align*}
\Bigl  (\bigl ((g')^* \circ (\, {}^*g) \bigr )\circ \frak F \Bigr )  ([V \xrightarrow h X; E]) 
&= \bigl ((g')^* \circ (\, {}^*g) \bigr ) ([X \xleftarrow h V \xrightarrow {f \circ h} Y; E])\\
&= (g')^* ([X \xleftarrow h V \xrightarrow {f \circ h} Y; E]) \, {}^*g\\
& =[X' \xleftarrow {h' \circ \widetilde {g''}} V' \times_V V'  \xrightarrow {f' \circ h' \circ \widetilde {g''}} Y'; (\widetilde {g''})^*(g'')^*E]).
\end{align*}
Here we consider the following fiber square:
$$\CD
V' \times _V V' @> {\widetilde {g''}} >> V'\\
@V {\widetilde {g''}} VV @VV {g''} V\\
V'  @>> {g''} > V. \endCD
$$
Thus in general we have that $\frak F \circ g^* \not = \bigl ((g')^* \circ (\, {}^*g) \bigr )\circ \frak F$ in the above diagram (\ref{dia}).
\end{enumerate}
\end{rem}
%%%%%%%%%%
\section{Generalizations}\label{dis}
In this section we discuss briefly some possible generalized versions of Annala's bivariant derived algebraic cobordism and our bi-variant algebraic cobordism treated in the previous section. 

%%%%%%%%
\subsection{A ``generalized" bivariant derived algebraic cobordism  $\Omega^*_{\frak S}(X \to Y)$ associated to a system $\frak S$ of subgroups $\frak s(A)$ of $\mathcal M_*(A)$ }\label{gbdac}

%%%%%%%%%%%
For Annala's bivariand derive algebraic cobordism $\Omega^*(X \xrightarrow f Y):= \frac{\mathcal M^*(X \to Y)}{\langle \mathcal R^{\op{LS}} \rangle (X \xrightarrow f Y)}$, the bivariant ideal $\langle \mathcal R^{\op{LS}} \rangle (X \xrightarrow f Y)$ of $\mathcal M^*(X \to Y)$ is constructed by using the kernels $\op{Ker} \eta_A$ of the morphisms $\eta_A: \mathcal M^*(A \to pt) = \mathcal M_*(A) \to d \Omega (A)$.  The kernel $\op{Ker} \eta_A$ is a very special subgroup of $\mathcal M^*(A \to pt)$. So, instead of taking this very special subgroup, we can consider an arbitrary subgroup, putting aside the issue of what its geometric or algebraic or topological meaning is. Namely, if we let $\frak S$ be a family or system of subgroups $\frak s(A) \subset \mathcal M^*(A \to pt)$ where $A \in Obj(\Cal C)$. Then the bivariant ideal $\langle \mathcal R^{\frak S} \rangle (X \xrightarrow f Y)$ associated to the system  $\frak S$
 is defined to consist of linear combinations of elements of of the following form, as in \cite[Lemma 3.8]{An} (also see Proposition \ref{bi-ideal} and Remark \ref{bi-ideal-rem}), 
\begin{equation}\label{form-5}
h_* \bigl (\alp_0 \bullet g^*s \bullet \beta_0 \bigr)
\end{equation}
where $h:A'' \to X, g:B' \to pt, \alp_0 \in \mathcal M^{*}(A'' \to A')^+, \beta_0 \in \mathcal M^{*}(B' \to Y)^+$ and $s \in \frak s(A) \subset \mathcal M^*(A \to pt)$ are as in the following diagram
\[
\xymatrix{
& X \ar@/^12pt/[drr]^f \\
A'' \ar[r]^{\Maru{$\alp_0$}} \ar@/^10pt/[ru]^h & A' \ar[r]^{\Maru{$g^*s$}} \ar[d] & B' \ar[d]^g  \ar[r]^{\Maru{$\beta_0$}} & Y \\
& A  \ar[r]^{\maru{s}} & pt
} 
\]
$h:A'' \to X$ is a proper morphism. 
Then a ``generalized" bivariant derived algebraic cobordism  $\Omega^*_{\frak S}(X \to Y)$ associated to a system $\frak S$ of subgroups $\frak s(A)$ of $\mathcal M_*(A)$ is defined to be the following quotient
\begin{equation}
\Omega^*_{\frak S}(X \to Y):= \frac{\mathcal M^*(X \to Y)}{\langle \mathcal R^{\frak S} \rangle (X \xrightarrow f Y)}.
\end{equation}
If each $\mathcal M^*(A \to pt)$ contains the kernel $\op{Ker} \eta_A$, then we have the surjection 
$\Omega^*(X \to Y) \twoheadrightarrow \Omega^*_{\frak S}(X \to Y)$. 
If each $\mathcal M^*(A \to pt)$ is contained in the kernel $\op{Ker} \eta_A$, then we have the other-way surjection $\Omega^*_{\frak S}(X \to Y) \twoheadrightarrow  \Omega^*(X \to Y).$
%%%%%%%%%%%
\subsection{A naive ``algebraic cobordism" $\Omega_{prop-sm}(X \to Y)$ using \emph{proper} and \emph{smooth} morphisms.} 
Let us replace the class of quasi-smooth morphisms by the class of smooth morphisms in Annala's bivariant theory $\mathcal M^*(X \xrightarrow f Y)$
and denote the replaced one by $\mathcal M^*_{prop-sm}(X \to Y)$.
Of course, if $f:X \to Y$ is not surjective, then $\mathcal M^*_{prop-sm}(X \to Y)=0$ the trivial group. Note that this kind of thing \emph{does} happen in bivariant theory as pointed out in \cite[\underline{Remarks}, p.62]{FM}. 
Then we have the following natural inclusion, since a smooth morphism is a quasi-smooth morphism:
$$\iota :\mathcal M^*_{prop-sm}(X \to Y) \hookrightarrow \mathcal M^*(X \to Y)$$
Then the kernel of the following composite morphism
$$\mathcal M^*_{prop-sm}(X \to pt) \xrightarrow {\iota} \mathcal M^*(X \to pt) = \mathcal M_*(X) \to d \Omega (X)$$
is denoted by $\mathcal R_{prop-sm}^{\op{LS}}(X \to pt)$.
Then, in the same way as the construction of Annala's derived bivariant algebraic cobordism $\Omega^*(X \to Y)$, we define 
$\Omega^*_{prop-sm}(X \to Y)$ as follows:
The bivariant ideal $\langle \mathcal R^{\op{LS}}_{prop-sm} \rangle (X \xrightarrow f Y)$
 is defined to consist of linear combinations of elements of the above form (\ref{form-5}) 
 where $s \in \mathcal R_{prop-sm}^{\op{LS}}(X \to pt)$.
Then we define 
$$\Omega_{prop-sm}^*(X \to Y):= \frac {\mathcal M^*_{prop-sm} (X \to Y)}{\langle \mathcal R^{\op{LS}}_{prop-sm} \rangle (X \xrightarrow f Y)}.$$
Here we emphasize that the category considered in this case is the classical category, not the category of derived schemes. Hence the explicit description (\ref{chern-explicit}) is not well-defined as an element $\Omega_{prop-sm}^*(X \xrightarrow {\op{id}_X} X)$ in the classical category, thus the Chern class operator $c_1(L) \bullet$ is not available\footnote{If we consider the universal oriented bivariant theory $\mathbb {OM}^{prop}_{sm}(X \to Y)$ \cite{Yokura-obt}, then the Chern class operator $c_1(L) \bullet$ is available.} in this case.
Thus $\Omega_{prop-sm}^*(X \to Y)$ is a bivariant theory without the data of the Chern class operator of line bundles.
It is easy to see that there is an embedding
$$\tilde {\iota}: \Omega_{prop-sm}^*(X \to Y) \hookrightarrow \Omega^*(X \to Y)$$
since we have $\iota \left (\langle \mathcal R^{\op{LS}}_{prop-sm} \rangle (X \xrightarrow f Y) \right ) \subset \langle \mathcal R^{\op{LS}} \rangle (X \xrightarrow f Y)$. 
%%%%%%%%%%%%%
%%%%%%%%%
\subsection{In the general case of $confined$ morphisms and $specialized$ morphisms or $vert$ morphisms and $horiz$ morphisms.}

Let $conf$ and $spe$ be the classes of confined morphisms and specialized morphisms, respectively and also $vert$ and $horiz$ be the classes of $vert$ morphisms and $horiz$ morphisms. Then, in the above section \S \ref{gnac}, $prop$-$sm$ can be replaced by $conf$-$spe$ and $vert$-$horiz$:
\begin{equation}
\Omega^*_{conf-spe,\frak S}(X \to Y), \qquad \Omega^*_{vert-horiz,\frak S}(X \to Y).
\end{equation}
Note that in a general case when there is no notion of line bundles or vector bundles available, we cannot consider the Chern class operators $c_1(L) \bullet$ of line bundles. Thus the above are bivariant theories without Chern class operators $c_1(L) \bullet$. However, here we \emph{emphasize} that in the category of derived schemes, if we let $confined$ and $specialized$ morphisms or $vert$ and $horiz$ morphisms be \emph{proper} morphisms and \emph{quasi-smooth} morphisms, respectively, and let $\frak S$ be the system consisting of the kernel $\op{Ker} \eta_A$ of the morphism $\eta_A: \mathcal M^*(A \to pt) \to d\Omega_*(A)$ for $A \in \Cal C$, then $\Omega^*_{conf-spe,\frak S}(X \to Y) =\Omega^*_{vert-horiz,\frak S}(X \to Y)$ is nothing but Annala's bivariant derived algebraic cobordism $\Omega^*(X \to Y)$, which \emph{is equipped with} the Chern class operators $c_1(L) \bullet$ of line bundles.
%%%%%%%%
 
\vspace{0.5cm}

{\bf Acknowledgements:}  
The author would like to thank J\"org Sch\"urmann and 
the anonymous referee for careful reading of the paper and for their many valuable comments and constructive suggestions. The author is supported by JSPS KAKENHI Grant Numbers JP19K03468 and JP23K03117. \\

\vspace{0.5cm}
%%%%%%%%%%%%%%%%%
%%%%%%%%%%%%%%%%%%%%%
%%%%%%%%%%%%%%%%%

\end{document}